\documentclass[10pt, a4paper]{amsart}
\usepackage[margin=3cm]{geometry}

\usepackage[T1]{fontenc}
\usepackage[utf8]{inputenc}
\usepackage{lmodern}
\usepackage[british]{babel}

\usepackage{csquotes}
\usepackage[
    sorting=nyt,
    doi=true,
    url=true,
    maxbibnames=4,
    backend=biber,
    giveninits=true,
    style=numeric,
    isbn=false,
    ]{biblatex}
\bibliography{literature}

\allowdisplaybreaks[4]

\usepackage{graphicx}
  \pdfpxdimen=\dimexpr 1 in/72\relax
\usepackage{subcaption}
\usepackage[dvipsnames]{xcolor}
  \definecolor{darkred}{RGB}{139,0,0}
  \definecolor{mediumblue}{RGB}{0,0,205}
  \definecolor{forestgreen}{RGB}{34,139,34}

\usepackage[colorlinks=true, hypertexnames=false]{hyperref}
\hypersetup{linkcolor=darkred, urlcolor=forestgreen, citecolor=mediumblue}

\usepackage{amsmath, amsfonts, amssymb}
\usepackage{mathtools}
\usepackage{commath}

\usepackage{bm}
\usepackage{stmaryrd}
\usepackage{etoolbox}
\preto\subequations{\ifhmode\unskip\fi}
\usepackage{nicefrac}
\usepackage{stackengine}
\usepackage{scalerel}

\usepackage{thmtools}
\makeatletter
\renewcommand\thmt@autorefsetup{\@xa\def\csname\thmt@envname autorefname\@xa\endcsname\@xa{\thmt@thmname}}
\makeatother

\usepackage[nameinlink]{cleveref}

\theoremstyle{plain}\newtheorem{thm}{Theorem}[section]
\newtheorem{lem}[thm]{Lemma}
\newtheorem{prop}[thm]{Proposition}
\newtheorem{cor}[thm]{Corollary}
\theoremstyle{definition}

\theoremstyle{remark}
\newtheorem{rmk}{Remark}

\crefname{thm}{Theorem}{Theorems}
\crefname{lem}{Lemma}{Lemmas}
\crefname{prop}{Proposition}{Propositions}
\crefname{cor}{Corollary}{Corollaries}

\usepackage{enumitem}
  \setlist[enumerate]{nosep, topsep=0pt, wide = 1em, leftmargin=*}
\usepackage{booktabs}
\AtBeginEnvironment{tabular}{\small}
\usepackage{siunitx}
\newcommand{\Oext}{{{\Omega}^\text{ext}}}
\newcommand{\OTt}{{\Omega_h^{\widetilde{\mathcal{T}}}}}

\newcommand{\Oh}{{\Omega_h}}
\newcommand{\Gh}{{\Gamma_h}}

\newcommand{\OT}{{\Omega_h^\mathcal{T}}}
\newcommand{\OGt}{{\widetilde\Omega_h^\Gamma}}
\newcommand{\OG}{\Omega_h^{\Gamma}}

\newcommand{\ThT}{\widetilde{\mathcal{T}}^{\rm ext}_h}

\newcommand{\Thm}{{\widetilde{\mathcal{T}}_{h}}}
\newcommand{\Thmgp}{\widetilde{\mathcal{T}}_{h}^{\text{gp}}}
\newcommand{\ThmG}{\widetilde{\mathcal{T}}_{h}^{\Gamma}}

\newcommand{\Th}{\widetilde{\mathcal{T}}_{h}^\textup{Al}}
\newcommand{\ThO}{\widetilde{\mathcal{T}}_{h}^{\textup{Al},\Omega}}
\newcommand{\ThG}{\widetilde{\mathcal{T}}_{h}^{\textup{Al},\Gamma}}

\newcommand{\Fh}{\widetilde{\mathcal{F}}_{h}^\textup{Al}}
\newcommand{\FhG}{\widetilde{\mathcal{F}}_{h}^{\textup{Al},\Gamma}}
\newcommand{\Fhgp}{\widetilde{\mathcal{F}}_{h}^{\textup{Al},\text{gp}}}

\newcommand{\Tho}{{\mathcal{T}}_{h}^\textup{Al}}
\newcommand{\ThoG}{{\mathcal{T}}_{h}^{\textup{Al},\Gamma}}
\newcommand{\Thogp}{\mathcal{T}_{h}^{\textup{Al},\text{gp}}}
\newcommand{\ThoO}{{\mathcal{T}}_{h}^{\textup{Al},\Omega}}

\newcommand{\Fhogp}{\mathcal{F}_{h}^{\textup{Al},\text{gp}}}

\newcommand{\ub}{\bm{u}}
\newcommand{\vb}{\bm{v}}

\newcommand{\xb}{\bm{x}}

\newcommand{\nb}{\bm{n}}
\newcommand{\nh}{\nb_h}

\newcommand{\fb}{\bm{f}}

\newcommand{\eb}{\bm{e}}

\newcommand{\uh}{\ub_h}
\newcommand{\vh}{\vb_h}

\newcommand{\ph}{p_h}
\newcommand{\qh}{q_h}
\newcommand{\pha}{p_h^\ast}
\newcommand{\qha}{q_h^\ast}

\newcommand{\lh}{\lambda_h}
\newcommand{\mh}{\mu_h}
\newcommand{\Lb}{\bm{L}}

\newcommand{\Vb}{\bm{V}}
\newcommand{\Vh}{\Vb_h}

\newcommand{\Qh}{Q_h}

\newcommand{\Sh}{\Sigma_h}

\DeclareMathOperator{\dist}{dist}
\DeclareMathOperator{\meas}{meas}

\newcommand{\jump}[1]{\llbracket#1\rrbracket}
\newcommand{\restr}[2]{{\left.\kern-\nulldelimiterspace#1\right|_{#2}}}

\newcommand{\nrm}[1]{\Vert #1 \Vert}

\newcommand{\tnrm}[1]{\mathopen{|\mkern-1.5mu|\mkern-1.5mu|}#1\mathclose{|\mkern-1.5mu|\mkern-1.5mu|}}
\newcommand{\tnrma}[1]{\tnrm{#1}_\ast}

\newcommand{\RR}{\mathbb{R}}
\newcommand{\PP}{\mathbb{P}}
\newcommand{\bV}{{\bm V}}
\newcommand{\bH}{{\bm H}}
\newcommand{\bv}{{\bm v}}
\newcommand{\nab}{\nabla}
\newcommand{\calT}{\mathcal{T}}
\newcommand{\bPhi}{{\bm \Phi}}
\newcommand{\bL}{{\bm L}}
\newcommand{\bbP}{\mathbb{P}}
\newcommand{\bn}{{\bm n}}

\newcommand{\p}{\partial}
\newcommand{\bbR}{\mathbb{R}}
\newcommand{\ba}{{\bm a}}
\newcommand{\bb}{{\bm b}}
\newcommand{\bp}{{\bm p}}
\newcommand{\calF}{\mathcal{F}}
\newcommand{\bw}{{\bm w}}
\newcommand{\bu}{{\bm u}}
\newcommand{\bX}{{\bm X}}

\newcommand{\bPsi}{{\bm \Psi}}

\title[Unfitted Scott--Vogelius Elements]{An unfitted divergence-free higher order finite element method for the Stokes problem}
\date{\today}

\author[M.~Neilan]{Michael Neilan}
\address{Department of Mathematics, University of Pittsburgh, Pittsburgh, PA 15260 USA}
\email{neilan@pitt.edu}

\author[M.~Olshanskii]{Maxim Olshanskii}
\address{Department of Mathematics, University of Houston, TX, USA}
\email{maolshanskiy@uh.edu}

\author[H.~v.~Wahl]{Henry von Wahl}
\address{Institute for Mathematics, Friedrich Schiller University Jena, Germany}
\email{henry.von.wahl@uni-jena.de}

\thanks{M.N.~was supported in part by the NSF, grant DMS-2309425. M.O.~was supported in part by the NSF, grant DMS-2408978.}

\begin{document}
\begin{abstract}
The paper develops and analyzes a higher-order unfitted finite element method for the incompressible Stokes equations, which yields a strongly divergence-free velocity field up to the physical boundary. The method combines an isoparametric Scott--Vogelius velocity–pressure pair on a cut background mesh with a stabilized Nitsche/Lagrange multiplier formulation for imposing Dirichlet boundary conditions. We construct finite element spaces that admit robust numerical implementation using standard elementwise polynomial mappings and produce exactly divergence-free discrete velocities. The key components of the analysis are a new inf-sup stability result for the isoparametric Scott--Vogelius pair on unfitted meshes and a combined inf-sup stability result for the bilinear forms associated with the pressure and the Lagrange multiplier. The finite element formulation employs a higher-order Lagrange multiplier space, which ensures stability and mitigates the loss of pressure robustness typically associated with the weak enforcement of boundary conditions for the normal velocity component.

The paper provides a complete stability and convergence theory in two dimensions, accounting for the geometric errors introduced by the isoparametric approximation. The analysis shows optimal-order velocity convergence in both the $H^1$ and $L^2$ norms and establishes optimal $H^1$-convergence and nearly optimal $L^2$-convergence of a post-processed pressure. Numerical experiments illustrate and confirm the theoretical findings. 
\end{abstract}

\maketitle

\section{Introduction}
\label{sec:intro}
Geometrically unfitted finite element methods have become widely accepted as a strong alternative to classical approaches based on body-fitted meshes, particularly when the partial differential equations are posed on complex or deforming domains. In this paper, we focus on a stabilized Nitsche-type unfitted FEM, also known as Nitsche--XFEM or CutFEM~\cite{burman2012fictitious,BCH+14}. The method was originally introduced for elliptic interface problems but has since rapidly evolved and been extended in various directions, including parabolic and hyperbolic interface problems and PDEs posed on time-dependent domains; see the recent review~\cite{burman2025cut}.

Among other applications, the development of CutFEM formulations for viscous incompressible fluid flow has been the subject of numerous studies (see, e.g.,~\cite{BH14,MLLR14,hansbo2014cut,kirchhart2016analysis,guzman2018inf,BMV18,massing2018stabilized,BFM22,vWRL20,olshanskii2025stability}). However, the unfitted finite element methods proposed in these works generally fail to produce a strongly solenoidal velocity field---an important property for ensuring that the discrete solution satisfies the correct physical balance laws. There are at least two reasons for this limitation. First, most popular finite element spaces are originally designed to enforce the incompressibility constraint only weakly. Second, the ghost-penalty terms introduced in CutFEM to guarantee algebraic stability typically perturb the divergence constraint.

For the reasons outlined above, the development of \textit{unfitted divergence-free} finite elements has only recently been addressed in several studies in the context of the Darcy and Stokes equations. The Darcy interface problem was addressed in \cite{FHNZ24} using $\mbox{RT}_k\times Q_k$ pairs together with a special pressure stabilization that guarantees pointwise divergence-free approximations and optimal convergence. The subsequent work \cite{frachon2024stabilized} extended these ideas to the Darcy problem with unfitted outer boundaries by employing a stabilized Lagrange multiplier formulation for essential boundary conditions, although only the lowest-order case was analyzed. A closely related method that also included a post-processing step for the scalar variable was introduced and analyzed in~\cite{LvBV23}.  All these  studies assumed exact integration over cut elements.

For the Stokes problem, \cite{LNO23} proposed a CutFEM based on the Scott--Vogelius pair, ensuring both finite element and algebraic stability through ghost-penalty terms for velocity and pressure, and achieving a discrete velocity field that is divergence-free outside an $O(h)$ neighborhood of the boundary. The study \cite{FNZ23} introduced two H(div)-conforming CutFEM formulations employing Brezzi--Douglas--Marini and Raviart--Thomas spaces, which produce pointwise divergence-free velocities and well-conditioned systems independent of the boundary position, though without providing a theoretical error analysis and assuming exact numerical integration over cut elements. The authors of \cite{BHL24} presented a low-order divergence-free CutFEM for the Stokes problem, employing two variants of a stabilized Lagrange multiplier method to impose Dirichlet conditions. The paper established optimal error estimates independent of cut geometry; however, the study in \cite{BHL24} was limited to low-order discretizations and assumed exact integration over cut elements.

Thus, a provably stable and higher-order accurate CutFEM for the Stokes problem that yields a strongly divergence-free solution up to the boundary remains an open challenge. The present study addresses this challenge with a two-dimensional isoparametric  Scott--Vogelius unfitted finite element method, in which the discrete velocity space consists of (mapped) piecewise polynomials of degree $k$, and the discrete pressure space consists of (mapped) piecewise polynomials of degree $k-1$.

The method introduced here benefits from several recent advancements in handling immersed and curvilinear interfaces within a finite element framework, namely: a Nitsche/Lagrange multiplier formulation from \cite{BHL24}, an isoparametric approximation technique for level set domains \cite{Leh16}, and a modification of the isoparametric Scott--Vogelius element from \cite{NO21,DN24}. The end result is a higher-order, stable, and optimally convergent scheme with an exactly divergence-free velocity approximation in both the physical and computational domains.

While these ingredients form the basis of the construction of the proposed discretization, the precise definition of our finite element spaces as well as the stability and convergence analysis requires new techniques beyond those found in the literature. First, our definition of the proposed Scott--Vogelius finite element space fundamentally differs from the approach given in \cite{NO21,DN24}.  The definition of the Scott--Vogelius finite element spaces given in those works are developed within a macro-element framework, with mapped reference basis functions defined on local Alfeld partitions. Although this viewpoint is advantageous for the stability analysis, it results in a technically involved and cumbersome implementation. Rather, we define the Scott--Vogelius space with respect to an isoparametric split mesh, mapped via standard polynomial basis functions. As a result, the proposed spaces can be implemented using existing finite element software packages such as {NGSolve} \cite{Sch14}. On the other hand, the stability proof of the proposed velocity-pressure finite element pair is considerably more involved and one of the focuses of this paper.

Second, we do not assume exact integration over cut elements and incorporate the geometric error in the scheme and analysis. Following the approach of \cite{Leh16}, we approximate the domain by a continuous, piecewise-polynomial diffeomorphism, which yields an approximate geometry with $O(h^{k+1})$ geometric error. This geometric perturbation enters the analysis in a nontrivial way, particularly in establishing the inf-sup stability of the Scott--Vogelius pair, as well as the derivation of optimal-order error estimates. 

Third, relative to \cite{BHL24}, we utilize a potentially higher-order Lagrange multiplier space for the pressure trace and use continuous, piecewise polynomials. We show that this choice, together with the Scott--Vogelius pair, yields a stable method. Moreover, the higher-order Lagrange multiplier space mitigates the lack of pressure robustness inherent in schemes with weakly imposed boundary conditions \cite{JLMNR17,BHL24,FNZ23}.

An alternative to using isoparametric finite elements in the CutFEM setting is to assume sufficiently accurate integration over implicitly defined cut portions of simplices and surfaces, using one of the dedicated techniques discussed, for instance, in \cite{saye2015high,SAYE2022110720,muller2013highly,garhuom2022non,olshanskii2016numerical}. Such techniques enable direct integration over cut elements, although each has its own limitations. If one employs one of these approaches and guarantees sufficiently accurate integration, possibly at the price of greater implementation effort and computational cost, then we expect that the main error estimates of the present paper can be extended to the corresponding variant. However, some details of the analysis would still need to be supplied.

Another variant of the method, not considered here, employs a Lagrange multiplier to enforce both the tangential and normal boundary conditions for the velocity; see \cite{BHL24} for the corresponding lower-order case.

\subsection{Outline of the paper}
The material in the paper is organized as follows. \Cref{sec:problem} formulates the Stokes problem and several integral identities satisfied by the solution, which we exploit to construct the finite element method. \Cref{sec:Prelim} introduces simplicial meshes, the necessary geometric preliminaries, and the finite element spaces. \Cref{sec:FE} presents the finite element formulation, consisting of the main scheme for approximating the fluid velocity, pressure {boundary} trace, and an auxiliary pressure, complemented by a post-processing step that recovers the physical pressure. We show that the finite element velocity is pointwise divergence-free.

\Cref{sec:SAnalysis} contains the stability analysis of the finite element method. It begins with the key inf-sup pressure stability result for the (isoparametric) Scott--Vogelius element, whose proof is provided in Appendix~\ref{sec:A}. The section then establishes the stability of the discrete Lagrange multiplier and, subsequently, the stability of the full bilinear form. Extending stability from the pressure variable to the entire formulation is nontrivial and requires careful scaling arguments. The error analysis is completed in \Cref{sec:EAnalysis}. The main theorem, \cref{thm:Main}, proves convergence of the finite element solutions to specific projections of the physical quantities onto the discrete spaces, from which the optimal $H^1$ error estimate for the velocity follows under sufficient regularity of the data. For the post-processed finite element pressure, \cref{Th:pressure} establishes optimal convergence in the $H^1$ norm and half-order suboptimal convergence in the $L^2$ norm.

Finally, \Cref{sec:Numerics} illustrates the analysis and assesses the performance of the isoparametric Scott--Vogelius element on a series of numerical examples. The results confirm the optimal order of accuracy for the finite element velocity and also suggest that the post-processed pressure is optimally accurate in the $L^2$ norm.

\section{The Stokes Problem}\label{sec:problem}
Let $\Omega$ be an open, bounded domain in $\RR^d$, $d\in\{2,3\}$, with a  smooth boundary $\Gamma=\partial\Omega$. The Stokes system  will serve here as a model problem: Find  a velocity field $\ub$ and pressure $p$, such that 
\begin{equation}\label{Stokes}
\left\{
\begin{aligned}
  - \Delta \ub + \nabla p &= \fb &&\text{in }\Omega\\
 \nabla\cdot\ub &= 0 &&\text{in }\Omega\\
   \ub &= 0 &&\text{on }\Gamma
\end{aligned}\right.
\end{equation}
holds. To avoid a non-unique pressure, we add $\int_\Omega p\dif \xb=0$. Assume that the ``physical'' domain $\Omega$ is embedded in an ambient polytopal ``computational'' domain  $\Oext$. 

For the purpose of convergence analysis, we assume that $\ub\in {\bH^m(\Omega):=}H^m(\Omega)^d$, $p\in H^{m-1}(\Omega)$ for a suitable positive exponent $m$.  Let $d=3$: Since $\ub$ vanishes on $\Gamma$ and is solenoidal, there exists a vector potential $\bv\in \bH^{m+1}(\Omega)$ such that $\ub=\nabla\times\bv$. One can extend $\bv$ from $\Omega$ to $\Oext$ in $\bH^{m+1}$~\cite{stein1970singular}, which provides us with an $\bH^m$-continuous solenoidal extension $\bu^e$ defined on  $\Oext$. A similar argument works for velocity in  $d=2$. For the pressure function we shall consider the extension by \emph{zero} outside $\overline{\Omega}$, which we denote by $p^z$.

Let us introduce the trace of the pressure on the boundary $\Gamma$ as an auxiliary variable {$\lambda = \restr{p}{\Gamma}$}.
It is easy to see that the extended solution satisfies the integral identity 
\begin{multline} \label{intIdent1}
  \int_\Omega\nabla\ub^e :\nabla\vb\dif\xb
  -\int_\Gamma ((\nb\cdot\nabla)\ub^e)\cdot\vb \dif s
  {-\int_\Gamma ((\nb\cdot\nabla)\vb)\cdot\ub^e \dif s}
  + \gamma\int_\Gamma \ub^e\cdot\vb\dif s\\
  -\int_{\Oext} p^z\nabla\cdot\vb\dif\xb  -\int_\Oext q\nabla\cdot\ub^e\dif\xb
  + \int_\Gamma \lambda (\vb\cdot\nb) \dif s
+ \int_\Gamma \mu (\ub^e\cdot \nb) \dif s = \int_\Omega\fb\cdot\vb\dif\xb
\end{multline}
for arbitrary and sufficiently regular test functions $\vb,q,\mu$ 
{defined on their corresponding domains, and for any real number $\gamma$.}
{Here, $\bn$ denotes the outward unit normal of $\Gamma$.}

Multiplying  the momentum equation  in \eqref{Stokes} by the gradient of $q$, it  also follows that the pressure solution satisfies the  identity
\begin{equation*}\int_\Omega\nabla p \cdot\nabla q\dif\xb = \int_\Omega(\Delta \ub)\cdot\nabla q\dif\xb + \int_\Omega\fb\cdot\nabla q\dif\xb.
\end{equation*} 
For vectors $\ba,\bb\in \bbR^d$, 
denote by $\ba\times \bb$ the cross product
in the case $d=3$, and $\ba\times \bb = \ba\cdot \bb^\perp = -\ba^\perp\cdot \bb$
in the case $d=2$, where $\ba^\perp = (-a_2,a_1)^T$.
Employing the identity for the vector Laplacian $\Delta(\cdot)=\nabla\nabla\cdot (\cdot) -\nabla\times\nabla\times (\cdot)$ followed by an integration by parts, we reformulate the pressure problem as 
\begin{equation}\label{intIdent2}
    \int_\Omega\nabla p \cdot\nabla q\dif\xb = -\int_\Gamma \bn\times(\nabla\times\ub)\cdot\nabla q\dif s + \int_\Omega\fb\cdot\nabla q\dif\xb.
\end{equation} 
The integral equalities \eqref{intIdent1} and \eqref{intIdent2} form the basis of the finite element formulation and the pressure post-processing step studied in this paper.

\section{Preliminaries} \label{sec:Prelim}

\subsection{Computational meshes}
Let $\ThT$ be a shape regular {and quasi-uniform} simplicial
 mesh of the background domain $\Oext$ {with ${\rm diam}(\tilde T) \approx h$ for all $\tilde T\in \ThT$}. The \emph{active \underline{macro} mesh} is then the set of elements with a contribution to the physical domain
\begin{equation*}
  \Thm \coloneqq \{T\in\ThT\mid \meas_{d}(T\cap\Omega) > 0\}.
\end{equation*}
The active domain is defined as the union of all simplices from $\Thm$: $\OTt \coloneqq \{\xb\in\RR^d \mid \exists T\in\Thm\text{ with }\xb\in T \}\subset \Oext$.  
The set of \emph{cut macro elements} is
\begin{equation*}
  \ThmG \coloneqq \{T\in\Thm \mid \meas_{d-1}(T\cap\Gamma) > 0 \},
\end{equation*}
and the cut simplices  together with their interior neighbors are collected in
\begin{equation*}
  \Thmgp = \{T\in\Thm \mid \exists\, T'\in\ThmG \text{ with }  \meas_{d-1}(T\cap T') >0 \}.
\end{equation*}

To define the unfitted Scott--Vogelius finite element pair, we refine the active macro mesh $\widetilde \calT_h$ by connecting the vertices with the element's barycenter (also known as an Alfeld or Clough-Tocher split), which we denote as $\Th$ and call it the \emph{active mesh}. In general, $\Th$ may contain elements that are outside the physical domain $\Omega$. Split elements that contain a non-empty intersection with the physical domain are collected in $\ThO \coloneqq \{K\in\Th\mid \meas_{d}(K\cap\Omega) > 0\}$.

Let $\Fh$ be the set of {$(d-1)$-dimensional} facets of the mesh $\Th$.
To stabilize the method with respect to arbitrary cuts between the geometry and the active mesh,
we define the set of \emph{ghost-penalty facets} as
\begin{equation}\label{eqn.gp-facets}
  \Fhgp \coloneqq \left\{ F \in \Fh \middle\vert \begin{array}{l} F=K_1\cap K_2,\; K_1, K_2\in\Th, K_1\neq K_2\\ \exists T_1,T_2\in\Thmgp\text{ with }K_1\subset T_1, K_2\subset T_2 \end{array}\right\}.
\end{equation}
To stabilize the Lagrange multiplier, we also collect the set of cut elements and facets of the split mesh
\begin{equation*}
  \ThG \coloneqq \{ K\in\Th\mid \meas_{d-1}(K\cap\Gamma) > 0\},
  \qquad
  \FhG \coloneqq \{ F\in\Fh \mid F\cap\Gamma\neq\emptyset \},
\end{equation*}
and define the resulting domain as $\OGt\coloneqq \{\xb\in\RR^d\mid \exists T\in \ThG\text{ with }\xb\in T\}$. We illustrate these different sets of elements in \Cref{fig.mesh-sketch}.

\begin{figure}
  \centering
  \includegraphics{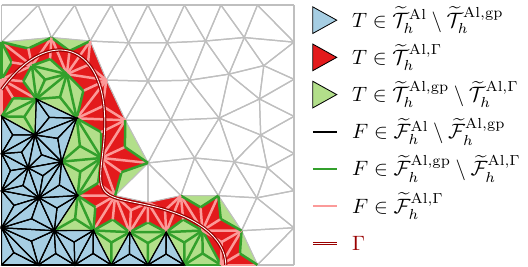}
  \caption{Sketch of element sets used in the unfitted Stokes discretization.}
  \label{fig.mesh-sketch}
\end{figure}

\subsection{Geometry Description}

We assume the physical domain is described by the sub-zero of a level set function $\phi:\RR^d\to\RR$, i.e., $\Omega = \{\xb\in\RR^d\mid \phi(\xb) < 0\},$ which is approximately a signed distance function. We shall use the same symbol $\bn$ to denote the normal on $\Gamma$ and its closest point extension to a neighborhood of $\Gamma$. 

On the discrete level, we have a finite element level set function $\phi_h$, which is  the continuous, piecewise linear interpolant of $\phi$. The corresponding polytopal domain is then given by
\begin{equation*}
  \widetilde{\Omega}_h = \{\xb\in\RR^d\mid \phi_h(\xb) < 0\},\qquad \widetilde{\Gamma}_h = \partial \widetilde{\Omega}_h.
\end{equation*}
We use this piecewise $\PP^1$ level set approximation  for robust construction of quadrature rules on the unfitted domain; see, e.g., \cite{BCH+14}. Since $\phi_h$ is (only) piecewise $\PP^1$, the discrete domain approximates the physical domain with an error
$
  \dist(\Gamma,\widetilde{\Gamma}_h) \lesssim h^2.
$
If left unaddressed, this geometric error dominates the approximation error of higher--order methods, resulting in suboptimal convergence rates.

Let $\bPhi_h:\,\Omega_h^{\tilde \calT}\to\Omega_h^{\tilde \calT}$ be the continuous, piecewise polynomial diffeomorphism  of degree $k$ with respect to $\Th$ constructed in \cite{Leh16}, and set
\begin{equation*}
  \Oh := \bPhi_h(\widetilde{\Omega}_h) = \{\bPhi_h(\xb)\mid\xb\in \widetilde{\Omega}_h\}\quad\text{and}\quad
  \Gh := \bPhi_h(\widetilde{\Gamma}_h) = \{\bPhi_h(\xb)\mid\xb\in\widetilde{\Gamma}_h\}. 
\end{equation*}
The geometric error of approximating $\Omega$ by $\Omega_h$ is characterized by 
\begin{equation*}
  \|\phi-\phi_h\circ\bPhi_h^{-1}\|_{W^{m,\infty}(\Omega_h^\Gamma)} \lesssim {h^{k+1-m}},\quad m=0,1.
\end{equation*}
In particular, this implies 
\begin{equation}\label{GeomError}
  \dist(\Gamma,\Gh) \lesssim h^{k+1}\quad \text{and}\quad \|\bn-\bn_h\|_{L^\infty(\Gamma_h)}\lesssim h^k,
\end{equation}
where $\bn_h$ is the outward unit normal of $\Gamma_h$. Since the mapping is piecewise polynomial with respect to $\Th$, one can apply standard quadrature formulas for (exact) numerical integration of finite element functions over $\Oh$ and $\Gh$.

\begin{rmk}
    Here and in what follows, we write $a\lesssim b$, iff there exists a constant $c>0$ which is independent of the mesh size and the geometry-mesh cut configuration, such that $a \leq c\,b$.
\end{rmk}

\subsection{Finite Element Spaces}

The Scott--Vogelius finite element pair $[\PP^k]^d \times \PP^{k-1}_\text{dc}$ is inf-sup stable on Alfeld-split affine meshes for $k \ge d$, and on affine meshes without nearly singular vertices for $k = 4$ and $d = 2$ \cite{Vog83, SV85, Qin94, Zha04}. Furthermore, on any affine (non-curvilinear) simplicial mesh, the divergence of the finite element velocity lies in the pressure space, yielding pointwise solenoidal discrete solutions of the finite element formulation. However, special care must be taken for curved meshes to preserve {these properties}. To this end, we consider a variation of the approach proposed in \cite{NO21, DN24}, which constructs
Scott--Vogelius pairs on isoparametric Alfeld splits. Similar
to the affine case and the references \cite{NO21, DN24},
we work on Alfeld splits and take the polynomial degree of the velocity space to be at least the spatial dimension, i.e., $k\ge d$.

Let $\Tho$ be the Alfeld split mesh curved by the mapping $\bPhi_h$, i.e.,
\[
\Tho := \bPhi_h(\Th)\coloneqq\{\bPhi_h(K) \mid K\in \Th\}.
\]
Likewise, we set $\ThoG = \bPhi_h(\ThG)$, $\ThoO=\bPhi_h(\ThO)$, $\Fhogp = \bPhi_h(\Fhgp)$,
$\Omega_h^\Gamma = \bPhi_h(\widetilde \Omega_h^\Gamma)$,
and 
\[
\OT = \{{\bm x}\in \mathbb{R}^d\mid \exists K\in \Tho\text{ with } {\bm x}\in K\}.
\]

We define the velocity space, for $k\ge d$, as follows: 
\begin{equation*}
  \Vh \coloneqq\left\{\vh\in\Lb^2(\OT)  \middle\vert \begin{array}{c} \restr{\vh}{ K} = \frac{1}{J_K}F_K\hat{\vb}_h\circ\varphi_K^{-1}\;\forall K\in\Tho,\;\hat{\vb}_h\in [\PP^{k}(\hat{T})]^d\\[1ex]
  \vh\text{ is single valued at Lagrange nodes of $\calT_h^{\textup{Al}}$}\end{array}\right\}.
\end{equation*}
Here, $\hat{T}$ is the reference element with unit length, $F_K = D\varphi_K$, $J_K=\det(F_K)$, where $\varphi_K: \hat{T}\rightarrow K$ is the mapping from the reference element to the physical element,
given by $\varphi_K = \bPhi_h \circ \varphi_{\tilde K}$, and $\varphi_{\tilde K}:\hat T\to \tilde K$
is affine, $K = \bPhi_h(\tilde K)$ with $\tilde K\in \Th$. 
If $\{\tilde {\bm a}\}$ denote the Lagrange nodes of the affine Alfeld mesh $\Th$,
then the Lagrange nodes of $\calT_h^{\textup{Al}}$ are defined as $\{\bPhi_h(\tilde {\bm a})\}$.
We take $\Vh^0$ to be the subspace of $\Vh$ consisting of functions vanishing on $\partial\OT$.

This construction differs from \cite{NO21,DN24}, as the latter considered the Piola mapping based on the \emph{macro} element. In contrast, we apply the deformation element-wise allowing an easier ``standard'' implementation via the reference element. Note that, due to the normal-preserving properties of the Piola transform, there holds $\Vh\subset {\bm H}({\rm div};\OT)$; see~\cite[Theorem 4.2]{NO21}. Moreover, the mapping $\bPhi_h$ reduces to the identity operator on all $K\in \calT_h^{\textup{Al}}$ such that $\bar K\cap \Omega_h^{\Gamma} = \emptyset$. Consequently, functions in $\bV_h$ are continuous, piecewise polynomials on the union of such elements.

For the pressure and Lagrange multiplier for the pressure trace, we consider the spaces
\begin{align*}
\Qh &= \{q_h\in L^2(\OT)\vert\ q_h|_K\circ \varphi_K\in \PP^{k-1}(\hat T)\ \forall K\in \Tho\},\\
\Sh &= \{\lambda_h\in C(\OG)\vert\ \lambda_h|_K\circ \varphi_K\in \PP^{k_\lambda}(\hat T)\ \forall K\in \ThoG\},\quad k_\lambda\in\{k-1,k\},
\end{align*}
and define $Q_h^0 = Q_h\cap L^2_0(\Omega_h^{\calT})$, the subspace of functions in $Q_h$ with vanishing mean.

The next result provides an approximation result of the discrete velocity space $\bV_h$. Its proof is given in the appendix.

\begin{lem}\label{lem:Interp}
There exists $I_V:{\bm C}(\Omega_h^{\calT})\to \bV_h$ such that for $s\in\{0,1\}$
\begin{equation}\label{eqn:InterpEstimate2}
\begin{aligned}
\|\bv -I_V\bv\|_{H^s(K)}&\lesssim h^{k+1-s}\|\bv\|_{H^{k+1}(K)}\quad &&\forall \bv\in \bH^{k+1}(K)\ \forall K\in \mathcal{T}^{\textup{Al}}_h,\\ 
\int_{\p \Omega_h^{\calT}} I_V \bv\cdot \bn_h\dif s &= \int_{\p \Omega_h^{\calT}} \bv\cdot \bn_h\dif s\qquad &&\forall \bv\in \bH^2(\Omega_h^{\calT}).
\end{aligned}
\end{equation}
\end{lem}

\section{Finite Element Method}\label{sec:FE}

We define the following bilinear forms:
\begin{align*}
  a_h(\uh, \vh) \coloneqq&
  \int_{\Oh} \!\! \nabla \uh :\nabla\vh\dif\xb
  -\!\int_{\Gh}\!\!(\nh\!\cdot\!\nabla \uh) \cdot\vh\dif s
  -\!\int_{\Gh}\!\!(\nh\!\cdot\!\nabla \vh) \cdot\uh\dif s
  +\frac{\gamma_n}{h}\!\int_{\Gh}\!\!\uh\cdot \vh \dif s, \\
   b_h(\qh,\vh) =& -\int_{\OT} \qh \nabla\cdot\vh\dif\xb,\\
   c_h(\mh,\vh) =& \int_{\Gh} \mh\nh\cdot\vh\dif s,
\end{align*}
for $\uh,\vh\in\Vh$, $\qh\in\Qh$, and $\mh\in\Sh$. Because $\bV_h \not\subset \bH^1(\OT)$ in general, throughout the paper the gradients of functions in $\bV_h$ are understood to be taken piecewise with respect to the mesh $\calT_h^{\textup{Al}}$. The same convention will be used for norms on $\bV_h$ as well, e.g., $\|\bv\|_{H^1(\Omega_h^{\calT})}^2 = \sum_{K\in \calT_h^{\textup{Al}}} \|\bv\|_{H^1(K)}^2$.

To stabilize the diffusion bilinear form with respect to arbitrary cuts between the geometry and the active mesh, we consider a ghost-penalty stabilization. Specifically, we consider the so-called \emph{direct version} of the ghost-penalty operator \cite{Pre18, LO19}. It was originally introduced for $H^1$-conforming finite elements on standard (affine) triangulations, and here we extend it to functions in $\bV_h$.

Let $\mathcal{E}^\mathcal{R}$ be the canonical global extension operator of rational functions to $\mathbb{R}^d$. In particular, if $\mathcal{E}^{\mathbb{P}}$ is the canonical global extension operator of polynomials to $\mathbb{R}^d$, and if $r = p/q$ for polynomials $p$ and $q$, then $\mathcal{E}^{\mathcal{R}} r = \mathcal{E}^{\mathbb{P}}p/\mathcal{E}^{\mathbb{P}}q$. For a facet $F= K_1\cap  K_2 \in \Fhogp$, $K_1\neq  K_2$, we define the \emph{facet patch} as $\omega_F =  K_1\cup  K_2$. The volumetric patch jump $\jump{\cdot}_\omega$ of
\begin{equation}\label{patch_jump}
  \restr{\jump{\vh}_\omega}{K_i} = \restr{\vh}{K_i} - \mathcal{E}^\mathcal{R}(\vh\circ\bPhi_h|_{\tilde K_j})\circ\bPhi_h^{-1},\quad\text{for }i,j\in\{1,2\}\text{ and }i\neq j.
\end{equation}
The direct ghost-penalty operator is then defined as
\begin{equation}\label{stab_form}
  i_h(\uh, \vh) \coloneqq \gamma_\textup{gp}\sum_{F\in\Fhogp}\frac{1}{h^2}\int_{\omega_F}\jump{\uh}_\omega\cdot\jump{\vh}_\omega\dif \xb.
\end{equation}

In our analysis, we need to extend the definition of $i_h(\ub,\vb)$, for $\ub,\vb\in L^2(\Omega)$. In this case, $i_h(\ub,\vb):=i_h(\ub_p,\vb_p)$, where $\restr{\ub_p}{K}= \big(\tilde J_K^{-1}\Pi_{\tilde K}(\tilde J_K\ub\circ\bPhi_h)\big)\circ\bPhi_h^{-1}$, with $\tilde K = \bPhi_h^{-1}(K)$, $\tilde J_K=\det(D\restr{\bPhi_h}{\tilde K})$, and $\Pi_{\tilde K}$ is the $L^2(\tilde K)$-orthogonal projection into polynomials of degree $2k-1$.

To stabilize the Lagrange multiplier, we  also consider the ``normal volume'' ghost-penalty stabilization
\begin{equation}\label{stabLM}
 j_h(\lh,\mh) \coloneqq  {h \gamma_\lambda} \int_{\Omega^\Gamma_h}
 (\bn_h\cdot\nabla \lh) (\bn_h\cdot\nabla \mh)\dif \xb,
\end{equation}
{where  $\bn_h=\nabla(\phi_h\circ\bPhi_h^{-1})/|\nabla(\phi_h\circ\bPhi_h^{-1})|$ in $\Omega_h^\Gamma$ is the extension of the normal vector.}
We now have all ingredients to define the finite element method.

\begin{rmk}[Necessary stabilization facets]\label{rmk.gp-facets}
  The set $\Fhgp$ of ghost-penalty facets defined in \eqref{eqn.gp-facets} is sufficient for stability but not necessary. We can reduce this set by considering patches of elements. Here, we group cut elements and facet neighbors of cut elements into disjoint groups $\mathcal{T}_\omega$, such that each patch $\omega$ contains at least one uncut root element $T_\omega\in\Thm\setminus\ThmG$. For ghost-penalty stabilization, we can then consider the set of interior facets of the Alfeld split elements of the macro patches. This can be seen as a weak form of the aggregated finite element method \cite{BNV22}.
\end{rmk}

\subsection{Finite element formulation}
The finite element formulation builds on the integral identity \eqref{intIdent1}. This can be seen as using the Nitsche method to enforce the boundary condition for the velocity weakly and the Lagrange multiplier method for the no-penetration condition. The formulation reads  as follows: Find $(\uh, \ph, \lh)\in \Vh\times\Qh^0\times\Sh$ such that \begin{multline}\label{eqn.discr-prob}
A_h((\bu_h,p_h,\lambda_h),(\vh,\qh,\mu_h))
:=(a_h + i_h)(\uh,\vh) + b_h(\ph, \vh) + b_h(\qh, \uh) \\
   + c_h(\lh, \vh) + c_h(\mh, \uh) + j_h(\lh, \mh) = \fb_h(\vh)
\end{multline}
for all $(\vh, \qh, \mh)\in \Vh\times\Qh^0\times\Sh$. Here,  $\fb_h(\vh)=(\fb,\vh)_{\Omega_h}$, where the source term $\fb$ is smoothly extended to $\Omega_h$.

\begin{rmk}
    Our practical implementation of the zero-mean constraint in the test and trial pressure space consists of adding $(\ph, s)_{\OT}+(\qh, r)_{\OT}$ {($s,r\in \mathbb{R}$)} to the 
finite element method bilinear form. 
\end{rmk}

The finite element pressure $\ph$ solving \eqref{eqn.discr-prob} approximates the extended pressure function $p^z$ which is zero outside $\overline \Omega$, and therefore discontinuous across the physical boundary $\Gamma$. Since the mesh is not fitted to $\Gamma$,  this leads to a suboptimal  --- $O(h^{\frac12})$, in general --- convergence order for $\ph$. We will see later that, thanks to the strong divergence-free property of the method, this pressure error does not affect the optimal convergence order for the velocity. 

The following lemma states that \eqref{eqn.discr-prob}
yields exactly divergence-free approximations.

\begin{lem}\label{lem:DivFree}
If $\uh\in \bV_h$ satisfies \eqref{eqn.discr-prob}, then $\nab\cdot \uh = 0$ in $\Omega_h^{\calT}$.
\end{lem}
\begin{proof}
Taking $\bv_h=0$, $\mu_h = 0$ in \eqref{eqn.discr-prob} yields $b_h(q_h^0,\uh)=0$ for all $q_h^0\in Q_h^0$. Splitting $q_h\in Q_h$ into the zero mean component and a constant, $q_h=q_h^0+|\OT|^{-1}(q_h,1)_{\OT}$, we find $(q_h,\nabla\cdot\uh-\mbox{const})_{\OT}=0$, with const=$|\OT|^{-1}(\uh\cdot\bn_{\partial\OT},1)_{\partial\OT}$. By repeating the arguments from \cite{DN24}, we obtain  $\nab\cdot \uh=$const in $\OT$. However, $c_h(\mu_h,\uh)=0$ for $\mu_h=1$ along with the divergence theorem implies that const$=0$.
\end{proof}

\subsection{Pressure post-processing}
We now benefit from this optimal-order finite element velocity in recovering the higher-order pressure approximation through a simple post-processing. For this, we also need the $H^1$-conforming pressure FE space:
\[
  \Qh^* = \{q_h\in C(\OT)\vert\ q_h|_K\circ \varphi_K\in \PP^{k-1}(\hat T)\ \forall K\in \Tho\}.
\]

The finite element formulation for the post-processing step builds on \eqref{intIdent2}: Find $\pha\in \Qh^*$ such that $(\pha,1)_{\Oh}=0$ and
\begin{align}\label{Post-process}
  (\nabla\pha,\nabla\qha)_{\Oh} + i_h(\pha, \qha)
  = \fb_h(\nabla \qha) - (\bn_h\times(\nabla\times\uh), \nabla \qha)_{\Gh}
  \quad \forall\,\qha\in \Qh^*.
\end{align}

With a slight abuse of notation, $i_h(\cdot,\cdot)$ is the direct ghost-penalty operator defined similarly to \eqref{stab_form}, but  acting on functions in $Q_h^*$. Furthermore, with another abuse of notation, the ghost-penalty operator $i_h$ is only needed over facets $\{F\in\Fhogp \mid \exists K_1\in \ThoO, K_2\in\ThoG, K_1\neq K_2\text{ and } F=\partial K_1\cap \partial K_2 \}$.  
Note that the post-processing formulation  applied integration by parts to avoid the implementation of the Hessian matrix for the Piola-mapped elements.

\section{Stability Analysis} \label{sec:SAnalysis}

The following theorem is the crucial result for the stability of the finite element method.
\begin{thm}[Inf-sup Stability]\label{Th:fitted}
  For $h>0$ sufficiently small, there exists $\beta>0$, independent of $h$, such that
  \begin{equation}\label{infsup1}
    \beta\nrm{q_h}_{\OT} \leq \sup_{\vh\in\Vh^0\setminus\{0\}} \frac{b_h(\qh, \vh)}{\nrm{\vh}_{H^1(\OT)}}\qquad\forall\qh\in\Qh^0.
  \end{equation}
\end{thm}
\begin{proof}
    See Appendix~\ref{sec:A}.
\end{proof}

\subsection{Stability of the Unfitted Formulation}
Recall that we use the same symbol $\bn$ to denote the normal on $\Gamma$ and its closest point extension to a neighborhood of $\Gamma$. For the analysis, we consider the following mesh-dependent norms for the velocity, pressure, and Lagrange multiplier:
\begin{align*}
  \tnrma{\vh}^2 &\coloneqq {\sigma}(\nrm{\nabla\vh}_{\OT}^2 + h^{-1}\nrm{\vh}_{\Gh}^2) +
  {\|(\nb\cdot\nabla)(\vh\cdot\bn)\|_{\Omega_h^\Gamma}^2}
  +\|\nab\cdot \vh\|_{\Omega_h^{\calT}}^{2},\\
  \tnrma{\qh}^2 &\coloneqq \nrm{q_h}_{\OT}^2,\\
  \tnrma{\lh}^2 &\coloneqq {\sigma^{-1}}(h\nrm{\lh}_{\Gh}^2 + j_h(\lh,\lh)),
\end{align*}
where the arguments make it clear which norms are used. The scaling parameter $\sigma$ and the narrow-band term $\|(\nb\cdot\nabla)(\vh\cdot\bn)\|_{\Omega_h^\Gamma}^2$ are introduced for the purpose of analysis.  We assume $0<\sigma\le1$, sufficiently small but independent of any discretization parameters. The product norm is then 
\begin{align*}
  \tnrma{(\vh, \qh, \lh)}^2 &\coloneqq \tnrma{\vh}^2 + \tnrma{\qh}^2  + \tnrma{\lh}^2. 
\end{align*}
We will make frequent use of the  unfitted trace inequality
\begin{equation}\label{eqn.unfitted-trace}
  \nrm{\vb}_{\Gamma\cap K} \lesssim h^{-\frac 12}\nrm{\vb}_K + h^{\frac 12}\nrm{\nabla\vb}_{K}
\end{equation}
for all $\vb\in \bH^1(K)$, see \cite{HH02}. We also recall the narrow-band estimate~\cite{grande2018analysis}:
\begin{equation}\label{EstLam}
    \|\lh\|_{\Omega_h^\Gamma} \lesssim \sigma^{\frac12}\tnrma{\lh}\qquad \forall \lh\in \Sigma_h.
\end{equation}

\subsubsection{Stability of the Pressure}
\begin{lem}
  \label{lemma:inf-sup2}
  There exist $h_0,\beta_b^\ast >0$ independent of $h$ and $\sigma$, such that for all $h\leq h_0$, it holds
  \begin{equation}\label{infsup3ABC}
    \beta_b^\ast\tnrma{\qh} \leq \sup_{\vh\in\Vh\setminus\{0\}} \frac{b_h(\qh, \vh)}{\tnrma{\vh}}\qquad\forall\qh\in\Qh^0. \end{equation}
\end{lem}

\begin{proof} 
Thanks to the result in \cref{Th:fitted}, there is $\vh\in\Vh^0\subset \Vh$ such that 
\[
\beta \tnrma{\qh}^2 \leq b_h(\qh, \vh)\quad \text{and}\quad 
\|\vh\|_{H^1(\OT)} \le \tnrma{\qh}. 
\]
The Nitsche term in the definition of the norm $\tnrma{\vh}$ can be bounded  by using the 
 unfitted trace inequality \eqref{eqn.unfitted-trace}, inverse estimates, and local Friedrichs inequalities:  for all $K\in \ThoG$,
\begin{align*}
  \nrm{\vh}_{\Gh\cap K}
  \lesssim h^{-\frac 12}\nrm{\vh}_K + h^{\frac 12}\nrm{\nabla\vh}_{K}
  &\lesssim h^{-\frac 12}\nrm{\vh}_K 
  \lesssim h^{\frac 12} \nrm{\nabla\vh}_{K}+h^{\frac{d-1}{2}}|\vh(\bp)|,
\end{align*}
where $\bp$ is any vertex of $K$. Now, $K$ belongs to a macro-element with a face on $\partial\OT$. Therefore, there exists a vertex $\bp$ of $K$ such that  $\vh(\bp)=0$. Hence, we obtain the bound $\nrm{\vh}_{\Gamma_h\cap K}^2\lesssim h \nrm{\nabla\vh}_{K}^2$. Summing up over all $K\in \ThoG$ gives the  control over the Nitsche term in $\tnrma{\vh}$:
\begin{equation}\label{aux625}
   h^{-\frac12}\|\vh\|_{\Gamma_h} \lesssim \|\vh\|_{H^1(\OT)}.
\end{equation}
Using this and $\sigma\le1$, we obtain
\[
\tnrma{\vh}\lesssim \|\vh\|_{H^1(\OT)} \le \tnrma{\qh}.
\]
This proves \eqref{infsup3ABC}.
\end{proof}

\begin{lem}\label{lem:Fortin}
For every  $\bv\in \bH^2(\Omega^{\calT}_h)$, there exists $\pi_V \bv\in \bV_h$ 
with $b_h(q_h,\pi_V \bv)=b_h(q_h,\bv)$
for all $q_h\in Q_h^0$.
If $\bv\in \bH^{k+1}(\Omega^{\calT}_h)$, then there holds $\tnrma{\bv - \pi_V \bv}\lesssim h^k \|\bv\|_{H^{k+1}(\Omega_h^{\calT})}$. 
Finally, if $\bv$ is divergence-free,
then $\nab\cdot \pi_V \bv \equiv 0$.
\end{lem}
\begin{proof}
Set $\bX_h = \{\vh\in \Vh^0:\ b_h(\qh,\vh) = 0\ \forall \qh\in \Qh^0\}$
to be the kernel of $b_h(\cdot,\cdot)$ and set
$\bX_h^\perp = \{\bw_h\in \Vh^0:\ (\nab \bw_h,\nab \vh)_{\Omega_h^{\calT}} = 0\ \forall \vh\in \bX_h\}$.
Given $\bv\in \bH^2(\Omega_h^{\calT})$, let $I_V \bv\in \Vh$ be the interpolant in \cref{lem:Interp}. 
Then by the inf-sup condition given in \cref{Th:fitted},
there exists a (unique) $\bw_h\in \bX_h^\perp$ satisfying
$b_h(q_h,\bw_h) = b_h(q_h,\bv-I_V \bv)$ for all $q_h\in Q_h^0$
and $\tnrma{\bw_h}\lesssim \|\bw_h\|_{H^1(\Omega_h^{\calT})}\lesssim \frac1{\beta}\tnrma{\bv-I_V \bv}$.
Setting $\pi_V \bv = \bw_h+I_V \bv$,
we see that $b_h(q_h,\pi_V \bv)=b_h(q_h,\bv)$ 
for all $q_h\in Q_h^0$, and
\[
\tnrma{\bv - \pi_V \bv}\le \tnrma{\bv - I_V\bv}+\tnrma{\bw_h}
\lesssim \left(1+ \beta^{-1}\right)\tnrma{\bv-I_V\bv}\lesssim h^k \|\bv\|_{H^{k+1}(\Omega_h^{\calT}))}.
\]
Finally, if $\bv$ is divergence-free,
then  $b_h(q_h,\pi_V \bv) = 0\ \forall q_h\in Q_h^0$, implying
 that $\nab\cdot \pi_V \bv$ is constant on $\Omega_h^{\calT}$.
Because $\bw_h\in \bV_h^0$, we have $\pi_V \bv|_{\p\Omega_h^{\calT}} = (\bw_h+I_V \bv)|_{\p\Omega_h^{\calT}} = I_V \bv|_{\p\Omega_h^{\calT}}$. We then conclude $\nab\cdot \pi_V \bv\equiv 0$
by \cref{lem:Interp} and the divergence theorem.
\end{proof}

\subsubsection{Stability of the Lagrange Multiplier}
\begin{lem}
  \label{lemma:inf-sup3}
  There exist $h_1,\beta^\ast_c >0$ independent of $h$ and $\sigma$, such that for all $h\leq h_1$, it holds 
  \begin{equation}\label{infsup3}
    \beta_c^\ast\tnrma{\mh} \leq \sup_{\vh\in\Vh\setminus\{0\}} \frac{c_h(\mh, \vh)}{\tnrma{\vh}} + \sigma^{-\frac12}j_h(\mh,\mh)^{\frac12}\qquad\forall\mh\in\Sh.
  \end{equation}
\end{lem}
\begin{rmk}
We note that for the proof that follows, it is important that the polynomial degree of the Lagrange multiplier does not exceed the degree of the finite element velocity space.
\end{rmk}
\begin{proof} 
Fix $\mu_h\in \Sigma_h$ and extend $\mu_h$ to $\Omega_h^{\calT}$ by setting its values at all Lagrange nodes in $\Omega_h^{\calT}\backslash \Omega_h^{\Gamma}$ to zero. The geometric error bound $\|\nb-\nb_h\|_{L^\infty(\Gamma_h)}\le c\,h^k$ implies $|1-\nb_h^T\nb|\le c h^{2k}$ on $\Gamma_h$. Therefore,
\begin{equation}\label{aux623}
  c_h(\mh, \mh\nb)\ge (1-ch^{2k})\|\mh\|^2_{\Gamma_h}.
\end{equation}
Since $\mh\nb$ is continuous, its nodal interpolant $\vh=I_V(\mh\nb)\in \bV_h$ given in \cref{lem:Interp} is well-defined.  

Thanks to the approximation properties of the interpolant, trace inequality \eqref{eqn.unfitted-trace},  inverse inequality  and \eqref{EstLam}, we have
\begin{equation}\label{aux635}
\begin{split}
    \|\vh- \mh\nb\|_{\Gamma_h}^2 &\lesssim h^{-1} \|\vh- \mh\nb\|_{\Omega_h^\Gamma}^2+ 
    h \sum_{K\in \ThoG} \|\nabla(\vh- \mh\nb)\|_{K}^2\\
    &\lesssim h^{2k+1} \sum_{K\in \ThoG}\|\mh\nb\|_{H^{k+1}(K)}^2 \lesssim h^{2k+1} \sum_{K\in \ThoG}\|\mh\|_{H^{k+1}(K)}^2\\ 
    &\lesssim h^{2k+1} \sum_{K\in \ThoG}\|\mh\|_{H^{k}(K)}^2\lesssim h \sum_{K\in \ThoG}\|\mh\|_{K}^2 = h\|\mh\|_{\Omega_h^\Gamma}^2
    \lesssim h\,\sigma\tnrma{\mh}^2.
\end{split}
\end{equation}
In the chain of inequalities above we used that $\|\mh\|_{H^{k+1}(K)}\lesssim \|\mh\|_{H^{k}(K)}$ for $\mh\circ\varphi_K \in \PP^{k}(\hat{T})$. The same arguments as above can be used to estimate $\nabla(\vh- \mh\nb)$ in the whole computational domain:
\begin{equation}\label{aux694}
\sum_{K\in \calT_h^{\textup{Al}}} \|\nabla(\vh- \mh\nb)\|_{\OT}^2 \lesssim \|\mh\|^2_{K}\lesssim \|\mh\|_{\Omega_h^\Gamma}^2
\lesssim \sigma\tnrma{\mh}^2.
\end{equation}
Here, the product $\mu_h \bn$ is understood to be the zero function outside an $O(h)$ neighborhood of $\Gamma_h$.

To bound $\tnrma{\vh}$ from above, we split 
\begin{equation*}
\tnrma{\vh}^2
\le \sigma \big(h^{-1}\|\vh\|_{\Gamma_h}^2+\|\nabla\vh\|_{\Omega_h^\Gamma}^2+
 \|\nabla\vh\|_{\OT\setminus\Omega_h^\Gamma}^2\big)
+ \|(\nb\cdot\nabla)(\vh\cdot\bn)\|_{\Omega_h^\Gamma}^2+\|\nab\cdot \vh\|_{\Omega_h^{\calT}}^2.
\end{equation*}
To bound the first two terms, we use  the triangle inequality followed by \eqref{aux635} and \eqref{EstLam},
\begin{equation}\label{aux645}
\begin{split}
h^{-1}\|\vh\|_{\Gamma_h}^2+\|\nabla\vh\|_{\Omega_h^\Gamma}^2
&\lesssim h^{-1}\|\mh\nb\|_{\Gamma_h}^2+\|\nabla(\mh\nb)\|_{\Omega_h^\Gamma}^2 \\
&\quad + h^{-1}\|\vh- \mh\nb\|_{\Gamma_h}^2+\|\nabla(\vh- \mh\nb)\|_{\Omega_h^\Gamma}^2\\
&\lesssim h^{-1}\|\mh\nb\|_{\Gamma_h}^2+\|\nabla(\mh\nb)\|_{\Omega_h^\Gamma}^2+\sigma\tnrma{\mh}^2\\
&\lesssim h^{-1}\|\mh\|_{\Gamma_h}^2+\|\mh\|_{H^1(\Omega_h^\Gamma)}^2+\sigma\tnrma{\mh}^2\\
&\le h^{-1}\|\mh\|_{\Gamma_h}^2+h^{-2}\|\mh\|_{\Omega_h^\Gamma}^2+\sigma\tnrma{\mh}^2
\le \sigma h^{-2}\tnrma{\mh}^2.
\end{split}
\end{equation}
Let $S_h=\partial(\OT\setminus\Omega_h^\Gamma)$. Recalling  $\vh=0$ at all internal nodes of 
$\OT\setminus\Omega_h^\Gamma$, we have $\|\nabla\vh\|_{\OT\setminus\Omega_h^\Gamma}^2\lesssim 
h^{-1}\|\vh\|_{S_h}^2$. In turn, by a standard finite element trace inequality, we have 
$\|\vh\|_{S_h}^2\le h^{-1}\|\vh\|_{\Omega_h^\Gamma}^2$.
We also estimate $\|\vh\|_{\Omega_h^\Gamma}^2\lesssim h\|\vh\|_{\Gamma_h}^2+ h^2\|\nabla \vh\|_{\Omega_h^\Gamma}^2$.
Putting these estimates together, we have 
\begin{align*}
\|\nabla\vh\|_{\OT\setminus\Omega_h^\Gamma}^2
&\lesssim h^{-1} \|\vh\|_{S_h}^2 \lesssim h^{-2} \|\vh\|_{\Omega_h^{\Gamma}}^2\lesssim h^{-1} \|\vh\|_{\Gamma_h}^2 + \|\nab \vh\|_{\Omega_h^\Gamma}^2.
\end{align*}

Combining this with \eqref{aux645} gives
\begin{equation}\label{aux629}
     h^2(\|\nabla\vh\|_{\OT}^2+h^{-1}\|\vh\|_{\Gamma_h}^2)\lesssim \sigma\tnrma{\mh}^2.
\end{equation}
We now estimate the narrow-band part of the $\tnrma{\vh}$ norm with the help of \eqref{EstLam}, \eqref{aux635}, and the definition of the $j_h(\cdot,\cdot)$ stabilization form: 
\begin{equation}\label{aux743}
\begin{split}
\|(\nb\cdot\nabla)(\vh\cdot\bn)\|_{\Omega_h^\Gamma}^2 &\lesssim \|(\nb\cdot\nabla)(\mh\nb\cdot\nb)\|_{\Omega_h^\Gamma}^2
+ \|\mh\nb-\vh\|_{H^1(\Omega_h^\Gamma)}^2\\ 
& \lesssim \|\nb\cdot \nabla \mh\|_{\Omega_h^\Gamma}^2 + \sigma\tnrma{\mh}^2
\lesssim \|\nb_h\cdot \nabla \mh\|_{\Omega_h^\Gamma}^2+h^{2k}\|\nabla \mh\|_{\Omega_h^\Gamma}^2 + \sigma\tnrma{\mh}^2\\ 
&\lesssim h^{-1}j_h(\mh,\mh) + \sigma\tnrma{\mh}^2 \lesssim \sigma h^{-1}\tnrma{\mh}^2.
\end{split}
\end{equation}
By the same arguments and \eqref{aux694}, we estimate the last (divergence)  term in the definition of the $\tnrma{\vh}$ norm: 
\[
\begin{split}
\|\nabla\cdot\vh\|_{\OT}^2 &\le\|\nabla\cdot(\mh\nb)\|_{\OT}^2 
+ C\|\nabla(\mh\nb-\vh)\|_{\OT}^2\\ 
&\lesssim\|\nabla\cdot(\mh\nb)\|_{\OT}^2  
+\sigma\tnrma{\mh}^2\\ 
&\lesssim\|\nb\cdot\nabla\mh\|_{\OT}^2 + \|\mh\|_{\OT}^2  
+\sigma\tnrma{\mh}^2\\ 
&\lesssim h^{-1}j_h(\mh,\mh) + \|\mh\|_{\Omega_h^\Gamma}^2 +\sigma\tnrma{\mh}^2
\lesssim \sigma h^{-1}\tnrma{\mh}^2.
\end{split}
\]

Combining this with \eqref{aux629}, \eqref{aux743}, and assuming $h$ is sufficiently small, i.e. $h\le \sigma$ holds,  we have
\begin{equation}\label{aux712}
     h\tnrma{\vh}\lesssim \sigma\tnrma{\mh}.
\end{equation}

Now \eqref{aux623}, \eqref{aux635}, \eqref{aux712} and the definition of norms lead to \eqref{infsup3} with the help of a straightforward calculation for sufficiently small $h$:
\[
\begin{split}
c_h(\mh, \bv_h)\ge& (1-ch^{2k})\|\mh\|^2_{\Gamma_h}+ c_h(\mh, \bv_h-\mh\nb)
\ge c\|\mh\|^2_{\Gamma_h}- \frac12\|\bv_h-\mh\nb\|^2_{\Gamma_h}\\
\ge& 
c\|\mh\|^2_{\Gamma_h}- c_1 \sigma h\,\tnrma{\mh}^2.
\end{split}
\]
Multiplying both sides by $ \sigma^{-1} h$ and adding $ {\sigma^{-1}} j_h(\mh,\mh)$ gives
\[
{\sigma^{-1}} j_h(\mh,\mh) + c_h(\mh,\sigma^{-1}h\bv_h)\ge c\sigma^{-1}(h\|\mh\|^2_{\Gamma_h}+j_h(\mh,\mh)) - c_1 h^2\,\tnrma{\mh}^2
\ge c\tnrma{\mh}^2.
\]
Dividing both sides by  $\tnrma{\mh}$ and using both \eqref{aux712} and ${\sigma^{-1}} j_h(\mh,\mh)<\tnrma{\mh}^2$ leads to
\eqref{infsup3}.
\end{proof}

\subsubsection{Stability of the cumulative bilinear form}

We first show the stability of the  cumulative constraint bilinear form and then combine it with the well-known results for the stabilized diffusion term to show stability of the global system.

\begin{lem}
  \label{lemma:inf-sup4}
  Assume that $\sigma>0$,  $h_2>0$ are sufficiently small. Then there exists $\beta_g^\ast >0$ independent of $h$, such that for all $h\leq h_2$ it holds
  \begin{equation}\label{infsup4}
    \beta_g^\ast(\tnrma{\mh}+ \tnrma{\qh})  \leq \sup_{\vh\in\Vh\setminus\{0\}} \frac{b_h(\qh, \vh)+ c_h(\mh, \vh) }{\tnrma{\vh}} + \sigma^{-\frac12} j_h(\mh,\mh)^{\frac12} \end{equation}
  for all $\qh\in\Qh^0,\, \mh\in\Sh.$
\end{lem}
\begin{proof}
  The key observation is the following: If $\vh$ vanishes on the boundary of $\OT$, then $\vh\cdot\bn_h$ in $c_h(\mh, \vh)$ can be properly controlled on  $\Gamma_h$ which is within an $O(h)$ distance from $\partial\OT$. More precisely, let $\vh^1\in\Vh^0$ be the (normalized) test function given by \cref{lemma:inf-sup2}, i.e. 
\[
\beta_b^\ast\tnrma{\qh} \leq b_h(\qh, \vh^1)\quad \text{and}\quad \tnrma{\vh^1}\le 1
\]
holds. Then, we estimate by using $\vh^1=0$ on $\partial\OT$,
\[
\begin{split}
c_h(\mh, \vh^1)&\le \|\mh\|_{\Gamma_h}\|\vh^1\cdot\bn_h\|_{\Gamma_h}
\le \|\mh\|_{\Gamma_h}\|\vh^1\cdot\bn\|_{\Gamma_h}+ h^{k}\|\mh\|_{\Gamma_h}\|\vh^1\|_{\Gamma_h}\\
&\le C\,h^{\frac12}\|\mh\|_{\Gamma_h}(\|(\nb\cdot\nabla)(\vh^1\cdot\bn)\|_{\Omega_h^\Gamma}+ h^{k-\frac12}\|\vh^1\|_{\Gamma_h})\\
&\le C\sigma^{\frac12}\tnrma{\mh}\tnrma{\vh^1}\le C\sigma^{\frac12}\tnrma{\mh}.
\end{split}
\]
Hence, by choosing a suitable $\sigma>0$, independent of $h$, one can make the factor $C\sigma^{\frac12}$ on the right-hand side sufficiently small.

The rest of the proof follows standard textbook arguments.
Using \eqref{infsup3}, there exists $\vh^2\in \bV_h$ satisfying $\beta_c^*\tnrma{\mh}\le c_h(\mh,\vh^2)+{\sigma^{-\frac12}} j_h(\mh,\mh)^{\frac12}$ and $\tnrma{\vh^2}\le 1$.
We then set $\vh = c_r \vh^1+\vh^2$ for some $c_r>0$ to be determined, and note that $\tnrma{\bv_h}\le 1+c_r$. We then find
\begin{align*}
b_h(q_h,\bv_h)+c_h(\mu_h,\bv_h) 
& = c_r b_h(q_h,\bv_h^1) +b_h(q_h,\bv_h^2) + c_r c(\mu_h,\bv_h^1) + c(\mu_h,\bv_h^2)\\
&\ge c_r \beta_b^* \tnrma{q_h} +\beta_c^*\tnrma{\mu_h} - {\sigma^{-\frac12}}j_h(\mu_h,\mu_h)^{\frac12} - |b_h(q_h,\bv_h^2)|-c_r |c(\mu_h,\bv_h^1)|\\
&\ge c_r \beta_b^* \tnrma{q_h}+(\beta_c^*-c_r C \sigma^{\frac12})\tnrma{\mu_h} - {\sigma^{-\frac12}} j_h(\mu_h,\mu_h)^{\frac12} - |b_h(q_h,\bv_h^2)|.
\end{align*}
The Cauchy--Schwarz inequality and the definition of the discrete norms yield $b_h(q_h,\bv_h^2)\le \tnrma{q_h}$. Thus,
\begin{align*}
b_h(q_h,\bv_h)+c_h(\mu_h,\bv_h) 
+{\sigma^{-\frac12}}j_h(\mu_h,\mu_h)^{\frac12}
\ge \big(c_r \beta_b^*-1\big) \tnrma{q_h}+(\beta_c^*-c_r C \sigma^{\frac12})\tnrma{\mu_h}.
\end{align*}
We now take $c_r$ sufficiently large and $\sigma$ sufficiently small (but independent of mesh parameters) to conclude
\begin{align*}
b_h(q_h,\bv_h)+c_h(\mu_h,\bv_h) 
+\sigma^{-\frac12}j_h(\mu_h,\mu_h)^{\frac12}
\ge \frac12 \min(\beta_b^*,\beta_c^*)(\tnrma{\mu_h}+\tnrma{q_h}).
\end{align*}
The desired result now follows from this estimate and $\tnrma{\bv_h}\le 1+c_r$.
\end{proof}

\subsubsection{Stability of the velocity}
To show the coercivity of the bilinear form $a_h(\cdot,\cdot)$ as well as to derive error estimates,  we need the following result, which is analogous to Lemmas~5.1 and~5.5 from \cite{LO19}. Its proof is given in the appendix.
\begin{lem}\label{LAux}
It holds
\begin{align} \label{aux702}
    \|\nab \vh\|_{K_1}^2\lesssim
    \|\nab \vh\|_{K_2}^2 &+ \frac{1}{h^2}\int_{\omega_F}\jump{\vh}_\omega^2\dif \xb\quad \forall\, F\in \Fhogp\qquad \forall \vh\in \bV_h,\\\label{aux705}
 i_h(\vb,\vb)&\lesssim h^{2k} \|\vb\|^2_{H^{k+1}(\Omega_h^\calT)}\\
 \label{aux707}
 i_h(\vb-\pi_V\vb,\vb-\pi_V\vb)&\lesssim h^{2k} \|\vb\|^2_{H^{k+1}(\Omega_h^\calT)}\quad\forall\,\vb\in \bH^{k+1}(\Omega_h^\calT).
\end{align}
\end{lem}

By a standard argument, see, e.g., \cite[Lemma 5.2]{LO19}, a repeated application of \eqref{aux702} yields 
\begin{equation}\label{ExtFE}
    \|\nab \bv_h\|^2_{\OT}\lesssim  \|\nab \bv_h\|^2_{\Omega_h} + i_h(\bv_h,\bv_h)\quad \forall\, \bv_h\in \bV_h.
\end{equation}
From \eqref{ExtFE}, the coercivity of $(a_h+i_h)(\cdot,\cdot)$ is immediate provided $\gamma_n$ is sufficiently large.

\begin{lem}
For $\gamma_n>0$ sufficiently large, there holds $\tnrma{\bv_h}^2\lesssim (a_h+i_h)(\vh,\vh)$ for all $\vh\in \bV_h$.
\end{lem}

\begin{thm}[Main stability result]\label{mainStab}
  For $\gamma_n>0$ sufficiently large and $\sigma$, $h>0$ sufficiently small, it holds for all $(\uh,\ph,\lh)\in\Vh\times\Qh^0\times\Sh$ that
  \begin{equation}\label{eqn:BigInfSup}
    \tnrma{(\uh,\ph,\lh)} \lesssim \sup_{\substack{(\vh,\qh,\mh)\in\Vh\times\Qh^0\times\Sh\\(\vh,\qh,\mh)\neq(0,0,0)}} \frac{A_h((\uh,\ph,\lh), (\vh,\qh,\mh))}{\tnrma{(\vh,\qh,\mh)}}.
  \end{equation}
\end{thm}
\begin{proof} The proof follows standard arguments, based on the coercivity of $(a_h+i_h)(\cdot,\cdot)$ on $\Vh$ for sufficiently large $\gamma_n$ and the inf-sup result \eqref{infsup4}.  See, e.g.,~\cite{BB03,ern2004theory} for the case $j_h=0$, and \cite{guzman2018inf} for the extension to $j_h\neq 0$. 
\end{proof}

\section{Error Analysis} \label{sec:EAnalysis}
Since \cref{mainStab} establishes  stability of the scheme, the finite element method \eqref{eqn.discr-prob} is well posed, and the error analysis reduces to the consistency of the scheme. Recall that the velocity space $\bV_h$ is $\bH({\rm div})$-conforming but not a subspace of $\bH^1$. Therefore, in addition to geometric consistency, we must also address the non-conformity of the velocity space. Here, we estimate this consistency error based on the results given in  \cite{DN24,NO21}. The arguments in these works are already highly technical in two dimensions and do not extend immediately to three dimensions.  For this reason, we restrict our error analysis to the 2D case.

The following lemmas estimate the consistency error arising from the non-conformity $\bV_h\not\subset \bH^1(\Omega^{\calT}_h)$.
\begin{lem}\label{lem:Consist}
Suppose  $d=2$ and the edge degrees of freedom 
of $\bV_h$ are the $\bPhi_h$-mapped Gauss--Lobatto nodes.
Then there holds, for all $\bw\in \bH^k(\Omega^{\calT}_h)$ and $\vh\in \bV_h$,
\begin{equation}\label{eqn:ConsistEst}
\sum_{F\in \calF_h^{\textup{Al},gp}} \int_F \left|\nabla \bw: [\vh]\right| \lesssim h^{k+m-1} \|\bw\|_{H^k(\Omega_h^{\calT})} \|\vh\|_{H^m(\Omega_h^{\calT})}\qquad m=0,1,\ldots,k,
\end{equation}
where $[\cdot]$ denotes the standard jump operator and 
$\|\vh\|_{H^m(\Omega_h^{\calT})}$ 
is understood to be the piecewise $H^m$ norm of $\vh$
with respect to $\calT_h^{\textup{Al}}$.
\end{lem}
\begin{proof}
In light of the assumptions on $\bPhi_h$, this result directly follows from the Poincaré inequality \cite[Lemma 4.7]{DN24}.
\end{proof}
{
\begin{rmk}
We note that \cref{lem:Consist} is the only point in the analysis that restricts the error estimates to the two-dimensional setting. If this consistency result also holds in three dimensions, then the error analysis developed below extends directly to that case. However, the proof of \cref{lem:Consist} relies crucially on high-order estimates for Gauss--Lobatto quadrature rules; see \cite[Lemma 4.7]{DN24}. Extending these arguments to three dimensions is not straightforward, due to the lack of an analogous Gauss--Lobatto quadrature rule on two-dimensional faces \cite{Helenbrook09}.
\end{rmk}
}

\begin{lem}\label{lem:Consist2}
Suppose  the assumptions in \cref{lem:Consist}
are satisfied. Let $\bu^e$ be the divergence-free
extension of the velocity solution, and let $\pi_V \bu^e\in \bV_h$ be
its approximation given in \cref{lem:Fortin}.
Then there holds
\begin{align}\label{eqn:Consistent2}
\left|(a_h+i_h)(\pi_V \bu^e,\vh) +\int_{\Omega_h} \Delta \bu^e\cdot \bv_h \right|
&\lesssim h^k \|\bu\|_{H^{k+1}(\Omega)} \tnrma{\vh}\qquad \forall \vh\in \Vh.
\end{align}
\end{lem}
\begin{proof}
First, by \eqref{aux705}--\eqref{aux707}, there
holds 
\[ i_h(\pi_V \bu^e,\bv_h)\lesssim h^k \|\bu^e\|_{H^{k+1}(\Omega_h^{\calT})}\tnrma{\vh}\lesssim h^k \|\bu\|_{H^{k+1}(\Omega)}\tnrma{\vh}.
\]

Next, we integrate by parts 
to obtain
\begin{align*}
&a_h(\pi_V \bu^e,\bv_h)+\int_{\Omega_h} \Delta \bu^e\cdot \bv_h\\
&\qquad = a_h(\pi_V\bu^e,\vh)-\int_{\Omega_h} \nab \bu^e:\nab \bv_h
+\int_{\Gamma_h} (\bn_h \cdot \nab \bu^e)\cdot \vh+\sum_{K\in \calT^{\textup{Al}}_h} \int_{\p K\cap \Omega_h} (\bn_{\p K} \cdot \nab \bu^e)\cdot \vh\\
&\qquad \le \int_{\Omega_h} \nab (\pi_V \bu^e-\bu^e):\nab \vh
+\int_{\Gamma_h} (\bn_h \cdot \nab (\bu^e-\pi_V \bu^e))\cdot \vh-\int_{\Gamma_h}(\bn_h\cdot \nab \vh)\cdot \pi_V \bu^e\\
&\qquad\qquad+ \frac{\gamma_n}{h} \int_{\Gamma_h} \pi_V \bu^e\cdot \vh
+\sum_{F\in \Fhogp} \int_F \left|\nabla \bu^e: [\vh]\right|.
\end{align*}
Denote by $\bp_\Gamma$ the closest point projection on $\Gamma$. Then, we have 
\[
h^{-1/2} \|\bu^e\|_{\Gamma_h}=h^{-1/2} \|\bu^e-\bu^e\circ\bp_\Gamma\|_{\Gamma_h}\lesssim h^{k+\frac12} \|\bu\|_{H^{k+1}(\Omega)},
\]
where  the last inequality follows from \cite[Lemma~7.3]{gross2015trace}.
We then apply standard trace inequalities, \cref{lem:Fortin}, Lemma
\ref{lem:Consist}, and the above estimate to conclude \eqref{eqn:Consistent2}.
\end{proof}

Using \cref{lem:Consist2} and arguments analogous to those in \cite[Theorem 5.2]{BHL24}, we derive error estimates in the energy norm. However, unlike the setting in \cite{BHL24}, we have $\nab \cdot \bV_h\not\subset Q_h$ and $\Gamma\neq\Gamma_h$. Consequently,  the arguments in \cite[Theorem 5.2]{BHL24} cannot be applied verbatim and must be adapted to account for these structural differences. In particular, we introduce a weighted $L^2$-projection to ensure a commuting property with respect to the bilinear form $b_h(\cdot,\cdot)$. We set $\pi_Q:L^2(\Omega^{\calT}_h)\to Q_h^0$ such that
\begin{equation}\label{eqn:piQ}
\sum_{K\in \calT_h^{\textup{Al}}} \int_K \frac{1}{J_K\circ \varphi_K^{-1}} (\pi_Q p)\qh\dif\xb = 
\sum_{K\in \calT_h^{\textup{Al}}} \int_K \frac{1}{J_K\circ \varphi_K^{-1}}  p\qh \dif\xb\qquad \forall \qh\in Q_h,
\end{equation}
where we recall $J_K = \det(D\varphi_K)$. Because $J_K\approx h^d$ (cf.~\eqref{eqn:varphiKBounds}), $\pi_Q$ is well-defined.

If $\bv\in \bV_h$ with $\bv|_K = \frac1{J_K} F_K \hat \bv \circ \varphi_K^{-1}$ (with $\hat \bv\in [\bbP_k(\hat T)]^d$), then $\nab \cdot \bv|_K = \frac1{J_K} \hat \nab \cdot \hat \bv\circ \varphi_K^{-1}$. Consequently, it follows from the definitions of $b_h(\cdot,\cdot)$ and $Q_h$ and \eqref{eqn:piQ} that
\begin{equation}\label{eqn:QL2}
b_h(\pi_Q p,\vh) = b_h(p,\vh)\qquad \forall \vh\in \bV_h.
\end{equation}
We also set $\pi_\Sigma:C(\Omega_h^\Gamma)\to \Sigma_h$ to be the Lagrange interpolant into $\Sigma_h$.

\def\eps{\varepsilon}
Recall that $\bu^e$ is a divergence-free extension
of the velocity and that $p^z$ is
the zero extension of $p$. To 
derive error estimates, we also let $p^e$ 
denote a smooth extension of $p$,
and then set $\bar p^e$ such that $\bar p^e|_{\overline{\Omega}_h} = p^e|_{\overline{\Omega}_h}$ and 
$\bar p^e=0$ on $\bbR^d\backslash \overline{\Omega}_h$.
For $\lambda= p |_{\Gamma}$
we  extend it to $\lambda^e:\Omega_h^\Gamma\to\RR$ via a normal extension.
Finally, we set $\fb^e = -\Delta \bu^e + \nab \bar p^e$.

Because $\Gamma$ is smooth, the following estimate for the Sobolev norms of the normal extension is well-known; see, e.g.,~\cite[Lemma~3.1]{reusken2015analysis}:
\begin{equation}\label{aux1049}
\|\lambda^e\|_{H^m(\Omega_h^\Gamma)}\lesssim h^{\frac12} 
\|\lambda\|_{H^m(\Gamma)} = h^{\frac12} 
\|p\|_{H^m(\Gamma)},
\end{equation}
for any fixed integer $m\ge0$ such that the norms above make sense.

\begin{thm}\label{thm:Main}
Assume that $\Gamma$ is smooth and that the assumptions in \cref{lem:Consist} are satisfied.
There holds
\begin{multline}\label{eqn:ThmEstimate}
\tnrma{(\pi_V \bu^e-\bu_h,\pi_Q \bar p^e- p_h,\pi_\Sigma \lambda^e - \lambda_h)}\\
 \lesssim h^k\|\bu\|_{H^{k+1}(\Omega)}+h^{k_\lambda+1} \|p\|_{H^{k_\lambda+2}(\Omega)}+\|\fb^e-\fb_h\|_{V_h^*},
\end{multline}
where $\|\fb\|_{V_h^*} = \sup_{\vh\in \Vh} (\fb,\vh)_{\Omega_h}/\tnrma{\vh}$,
and we recall that $k_\lambda\in \{k-1,k\}$.
\end{thm}
\begin{proof}
Set
\[
{\bm e}_h = \pi_V \bu^e-\bu_h,\qquad \omega_h = \pi_Q {\bar p^e}- p_h,\qquad \eta_h = \pi_\Sigma \lambda^e-\lambda_h.
\]
Using the inf-sup stability property \eqref{eqn:BigInfSup} and the fact that ${\bm e}_h$ is in the kernel of $b_h(\cdot,\cdot)$, there exists $(\bv_h,\mu_h)\in \bV_h\times \Sigma_h$ such that $\tnrma{(\bv_h,0,\mu_h)}=1$ and  $\tnrma{({\bm e}_h,\omega_h,\eta_h)}\lesssim A_h(({\bm e}_h,\omega_h,\eta_h),(\bv_h,0,\mu_h))$. Applying the definition of the finite element method,
 \eqref{eqn:QL2}, then followed by \cref{lem:Fortin},
 \eqref{eqn:InterpEstimate2}, and \eqref{eqn:Consistent2} yields
\begin{equation}\label{eqn:ThmStart}
\begin{aligned}
\tnrma{({\bm e}_h,\omega_h,\eta_h)}
&\lesssim A_h(\pi_V \bu^e,\pi_Q \bar p^e,\pi_\Sigma \lambda^e),
(\bv_h,0,\mu_h))
- \int_{\Omega_h}\!\! \fb^e\cdot \bv_h\dif\xb + \int_{\Omega_h}\!\!(\fb^e-\fb_h)\cdot \bv_h\dif\xb\\
& = (a_h+i_h)(\pi_V \bu^e,\bv_h)+\int_{\Omega_h} \Delta \bu^e \cdot \bv_h\dif\xb\\
&\qquad + b_h(\pi_Q \bar p^e,\bv_h) - \int_{\Omega_h} \nab \bar p^e\cdot \bv_h\dif\xb
+ \int_{\Omega_h} (\fb^e-\fb_h)\cdot \bv_h\dif\xb\\
&\qquad \qquad+c_h(\pi_\Sigma \lambda^e,\bv_h)+c_h(\mu_h,\pi_V \bu^e) + j_h(\pi_\Sigma \lambda^e,\mu_h)\\
& = \left((a_h+i_h)(\pi_V \bu^e,\bv_h)+\int_{\Omega_h} \Delta \bu^e \cdot \bv_h\dif\xb\right) + \int_{\Omega_h} (\fb^e-\fb_h)\cdot \bv_h\dif\xb
\\
&\qquad \qquad 
+c_h(\pi_\Sigma \lambda^e-\bar p^e,\bv_h)+c_h(\mu_h,\pi_V \bu^e) + j_h(\pi_\Sigma \lambda^e,\mu_h)\\
&\lesssim h^k \|\bu\|_{H^{k+1}(\Omega)} +\|\fb^e-\fb_h\|_{V_h^*}\\
&\qquad\qquad+c_h(\pi_\Sigma \lambda^e-{\bar p^e},\bv_h)+c_h(\mu_h,\pi_V \bu^e)+ j_h(\pi_\Sigma \lambda^e,\mu_h).
\end{aligned}
\end{equation}

Applying standard approximation properties of the Lagrange interpolant, the estimate $|\bn-\bn_h|\lesssim h^k$ and \eqref{aux1049}, we get
\begin{equation}\label{aux1116}
\begin{split}
j_h(\pi_\Sigma \lambda^e,\mh)
%&= j_h(\pi_\Sigma \lambda^e - \lambda^e,\mh)\\
& =h\gamma_\lambda(\bn_h\cdot \nab (\pi_\Sigma \lambda^e-\lambda^e),\bn_h\cdot \nab \mu_h)_{\Omega_h^\Gamma} + h((\bn_h-\bn)\cdot \nab \lambda^e,\bn_h\cdot \nab \mu_h)_{\Omega_h^\Gamma}\\
&\lesssim h^{\frac12}\left(\|\bn_h\cdot \nab (\pi_\Sigma \lambda^e-\lambda^e)\|_{\Omega_h^\Gamma} + h^k\| \nab \lambda^e\|_{\Omega_h^\Gamma}\right)\tnrma{\mu_h}\\
&\lesssim (h^{k_\lambda+\frac12}+h^{k+\frac12})\|\lambda^e\|_{H^{k_\lambda+1}(\Omega_h^\Gamma)}\\
&\lesssim (h^{k_\lambda+1}+h^{k+1}) \|p\|_{H^{k_\lambda+1}(\Gamma)} \lesssim h^{k_\lambda+1}\|p\|_{H^{k_\lambda+1}(\Gamma)}.
\end{split}
\end{equation}
We now split $c_h(\pi_\Sigma \lambda^e-{\bar p^e},\bv_h)=c_h(\pi_\Sigma \lambda^e-\lambda^e,\bv_h) + c_h(\lambda^e-\bar p^e,\bv_h)$.
For the first term, we have
\begin{equation}\label{aux1140}
c_h(\pi_\Sigma \lambda^e - \lambda^e,\vh)
\le h^{\frac12} \|\pi_\Sigma \lambda^e -\lambda^e \|_{\Gamma_h} \tnrma{\vh}
\lesssim h^{\frac12}h^{k_{\lambda}+\frac12}\|\lambda^e\|_{H^{k_\lambda+1}
(\Omega_h^\Gamma)} \lesssim h^{k_{\lambda}+\frac32}
\|p\|_{H^{k_\lambda+1}(\Gamma)}.
\end{equation}

The second term is bounded as follows:
\begin{equation}\label{aux1145}
c_h(\lambda^e-p^e,\bv_h)
\le h^{\frac12} \|\lambda^e-\bar p^e \|_{\Gamma_h} \tnrma{\vh} 
\lesssim h^{\frac12}\|\lambda^e-\bar p^e \|_{\Gamma_h} =h^{\frac12}\|p\circ \bp_\Gamma-p^e \|_{\Gamma_h}
\lesssim h^{k+\frac32}\|p\|_{H^2(\Omega)}.
\end{equation}
Above $\bp_\Gamma$ is the closest point projection on $\Gamma$ and for the last inequality, we used the result from \cite[Lemma~7.3]{gross2015trace}.  

Likewise, we bound
\begin{equation}\label{eqn:ThmEndEst}
c_h(\mu_h,\pi_V \bu^e)\le h^{-\frac12}\|\pi_V \bu^e\|_{\Gamma_h} \tnrma{\mu_h}
\le (\tnrma{\bu^e - \pi_V \bu^e}+h^{-\frac12} \|\bu^e\|_{\Gamma_h})\lesssim h^k \|\bu\|_{H^{k+1}(\Omega)}.
\end{equation}
Combining \eqref{eqn:ThmStart}--\eqref{eqn:ThmEndEst} yields the desired result \eqref{eqn:ThmEstimate}.\end{proof}

The right-hand side perturbation term $\|\fb^e-\fb_h\|_{V_h^*}$ in \eqref{eqn:ThmEstimate} can be non-zero since the extension of the body forces  $\fb$ used to define $\fb_h$ can be, in general, not the same as $\fb^e=-\Delta\bu+\nab \bar p^e$. However, this term is higher order if $\fb$ is sufficiently smooth and $\Omega$ is sufficiently regular. This can be formulated in the following proposition.

\begin{prop}\label{Prop1} Assume $\fb\in \bH^2(\Omega)$ and $\Gamma\in C^4$. Then it holds that
\[
\|\fb^e-\fb_h\|_{V_h^*}\le {h^{2k+\frac52}} \|\fb\|_{H^2(\Omega)}.
\]
\end{prop}
\begin{proof}
For the arguments below it is convenient to assume $\fb^e=-\Delta\bu+\nab p^e$, where $p^e$ is 
the smooth extension,  not clipped to zero  outside $\Omega_h$. Obviously, the assumption does not affect $\fb^e$ in $\Omega_h$ and so $\|\fb^e-\fb_h\|_{V_h^*}$ does not change. 

From the definition of the dual norm and the fact that $\fb^e-\fb=0$ in $\Omega$, we derive 
\[
(\fb^e-\fb,\vh)_{\Omega_h}/\tnrma{\vh}\le  \|\fb^e-\fb\|_{\Omega_h\setminus\Omega} \|\vh\|_{\Omega_h\setminus\Omega}/\tnrma{\vh}.
\]

Denote by $U_\eps(\Gamma)$ an $\eps$-tubular neighborhood of $\Gamma$ and let $\Omega_\eps=\Omega\cup U_\eps(\Gamma)$. For  $g\in H^1(\Omega)$ and its smooth extension $g^e$ one can estimate~\cite[Lemma~4.10]{elliott2013finite}:
\begin{equation}\label{est:narrow}
    \|g^e\|_{L^2(U_\eps(\Gamma))} \lesssim \eps^{\frac12}
    \|g^e\|_{H^1(\Omega_\eps)} \lesssim 
    \eps^{\frac12}
    \|g\|_{H^1(\Omega)}.
\end{equation}

We let  $\eps = c\,h^{k+1}$ with $c$ sufficiently large such that $\Omega_h\setminus\Omega\subset U_\eps(\Gamma)$. We first use the Poincaré inequality in $U_\eps(\Gamma)$ followed by the triangle and the narrow-band inequalities as well as  the regularity result for the solution of the Stokes equation:
\[
\begin{split}
\|\fb^e-\fb\|_{\Omega_h\setminus\Omega}&\le \|\fb^e-\fb\|_{U_\eps(\Gamma)}\\
&\lesssim h^{k+1} \|\nab(\fb^e-\fb)\|_{U_\eps(\Gamma)}\\
&\lesssim h^{k+1} (\|{\bm u}^e\|_{H^3(U_\eps(\Gamma))}+
\|p^e\|_{H^2(U_\eps(\Gamma))}+\|\nab\fb\|_{U_\eps(\Gamma)})\\
&\lesssim h^{3(k+1)/2} (\|{\bm u}^e\|_{H^4(\Omega_\eps)}+
\|p^e\|_{H^3(\Omega_\eps)}+\|\bm f\|_{H^2(\Omega_\eps)})\\
&\lesssim h^{3(k+1)/2} (\|{\bm u}^e\|_{H^4(\Omega)}+
\|p^e\|_{H^3(\Omega)}+\|\bm f\|_{H^2(\Omega)})
\lesssim h^{3(k+1)/2} \|\bm f\|_{H^2(\Omega)}.
\end{split}
\]
Finally, we apply \cite[Lemma~7]{lehrenfeld2018stabilized} to estimate $\|\bv_h\|_{\Omega_h\setminus\Omega}\lesssim h^{\frac{k}2}\|\bv_h\|_{\Omega_h^\Gamma}$,  and so by the Poincaré and the definition of the $\tnrma{\vh}$ norm, we have 
$\|\bv_h\|_{\Omega_h^\Gamma}\lesssim h\tnrma{\vh}$.
\end{proof}

Combining the result of the theorem with the interpolation property from \cref{lem:Fortin}, 
{along with standard estimates using the inverse inequality,}
we have the
convergence result for the finite element velocity:
\begin{cor}\label{mainCor}
    Assume the solution $\bu$, $p$ to \eqref{Stokes}, $\fb$, and $\Gamma$ are sufficiently smooth. Then it holds
    \[
    h^{\frac12}\|\nab (\bu^e - \bu_h)\|_{\Gamma_h}+ \tnrma{\bu^e-\bu_h} \le C h^k,
    \]
 where $C$ depends on the problem data,
 and we recall the gradient is understood
 to be taken piecewise with respect to $\calT_h^{\textup{Al}}$.
\end{cor}

{
\begin{rmk}\label{rmk.pressure-robustness}
The constant $C$ in the estimate of \cref{mainCor} depends on Sobolev norms of the pressure variable (cf.~\eqref{eqn:ThmEstimate}). Thus, despite yielding divergence-free velocity approximations, the error in velocity may depend on the quality of the pressure approximation. This lack of ``pressure robustness'' is due to the weak imposition of boundary conditions, and numerical experiments in \Cref{sec:noFlow} indicate that this pressure dependence is reflected in the velocity approximation.
\end{rmk}
}

We now prove an error estimate for the post-processed pressure.  Recall the pressure post-processing step (cf.~\eqref{Post-process}): 
\begin{equation*}
a_{p,h}(r_h,\phi_h^*)  =    \fb_h(\nabla \qha) - (\bn_h\times(\nabla\times\uh), \nabla \qha)_{\Gh}\quad \forall\,\qha\in \Qh^*,
\end{equation*}
with $a_{p,h}(\pha,\qha) = (\nabla\pha,\nabla\qha)_{\Oh} + i_h(\pha, \qha)$.

\begin{thm}\label{Th:pressure} 
  Assume the solution $(\bu$, $p)$ to \eqref{Stokes}, $\fb$, and $\Gamma$ are sufficiently smooth.  Let $p^e$ be a smooth extension of $p$ to $\RR^d$, and let $\ph^*$ be the solution to \eqref{Post-process}. There holds
\begin{equation}\label{pressureError}
h^{\frac12}\|\nabla(p^e-\ph^*)\|_{\Omega_h}+\|p^e-\ph^*\|_{\Omega_h}\le C\,h^{k-\frac12},
\end{equation}
where $C$ is a constant depending on the problem data.
\end{thm}
\begin{proof}
From \eqref{Post-process} we have
\begin{align*}
&(\nab (p^e-p_h^*),\nab \qh^*)_{\Omega_h}
+ i_h(p^e-p_h^*,\qh^*)\\
&\qquad = (\nab p^e,\nab \qh^*)_{\Omega_h} +i_h(p^e,\qh^*)
-(\fb,\nab \qh^*)_{\Omega_h} + (\bn_h\times (\nab \times \bu_h),\nab \qh^*)_{\Gamma_h}\\
&\qquad = i_h(p^e,q_h^*) + (\Delta \bu^e,\nab \qh^*)_{\Omega_h} - (\bn_h\times (\nab \times \bu_h),\nab \qh^*)_{\Gamma_h}+c_f(\qh^*),
\end{align*}
with $c_f(\qh^*)=(\fb^e-\fb,\nab \qh^*)_{\Omega_h}$.
Using $\Delta \bu^e = \nab \nab \cdot \bu^e - \nab\times \nab \times \bu^e = - \nab\times \nab \times \bu^e$, integration-by-parts, the continuity of $\qh^*$, 
and \cref{mainCor}, there holds
\begin{equation}
\label{orthog}
\begin{split}
&(\nab (p^e-p_h^*),\nab \qh^*)_{\Omega_h}
+ i_h(p^e-p_h^*,\qh^*)\\
&\quad = i_h(p^e,\qh^*) -\sum_{T\in \calT^{\textup{Al}}_h} \int_{\p (T\cap \Omega_h)} ((\nab \times \bu^e) \times \nab \qh^*)\cdot \bn_h  + (\bn_h\times (\nab \times \bu_h),\nab \qh^*)_{\Gamma_h}+c_f(\qh^*)\\
&\quad = i_h(p^e,\qh^*) + (\bn_h\times (\nab \times (\bu_h-\bu^e)),\nab \qh^*)_{\Gamma_h}+c_f(\qh^*)\\
&\quad \lesssim 
\left(i_h(p^e,p^e)^{\frac12} +h^{k-1} +\|\fb-\fb^e\|_{-1,h}\right) \|\qh^*\|_{1,h},
\end{split}
\end{equation}
where $\|\qh^*\|_{1,h}^2 = \|\nab \qh^*\|_{\Omega_h}^2 + i_h(\qh^*,\qh^*)$, and 
\[
\|\fb^e-\fb\|_{-1,h}=\sup_{q_h^*\in Q_h^*} \frac{(\fb^e-\fb,\nab \qh^*)_{\Omega_h}}{\|\qh^*\|_{1,h}}.
\]

By repeating the same arguments as used to prove \cref{Prop1}, we get
$\|\fb^e-\fb\|_{-1,h}\lesssim h^{2k+\frac32} \|\fb\|_{H^2(\Omega)}.$ Therefore
by \eqref{orthog} and Strang's lemma, we conclude
\begin{equation}\label{pH1}
\|p^e - p_h^*\|_{1,h}\lesssim h^{k-1}+h^{k_\lambda}\lesssim h^{k-1},
\end{equation}
which yields the optimal-order bound for $\|\nabla(p^e-\ph^*)\|_{\Omega_h}$ in \eqref{pressureError}.

Let $\bPsi_h:\Omega\to\Omega_h$ be a $W^{1,\infty}$ diffeomorphism such that 
\begin{equation}\label{diffEst}
    \|id-\bPsi_h\|_{L^\infty(\Omega)}+h\|{\bf I}-D\bPsi_h\|_{L^\infty(\Omega)} \lesssim h^{k+1},
\end{equation}
and denote $\mu_h=\det(D\bPsi_h)$. Such a diffeomorphism is constructed, for example in~\cite{gross2015trace,lehrenfeld2018analysis}. Moreover, for a smooth extension of $q\in H^1(\Omega)$, we have (cf.~\cite[Lemma 7.3]{gross2015trace}):
\begin{equation}\label{aux1306}
    \|q-q^e\circ\bPsi_h\|_{\Omega}\lesssim h^{k+1}\|q\|_{H^1(\Omega)}.
\end{equation}

Set $\rho = (p^e-\ph^*)\circ\bPsi_h$, and let $\tilde\rho=\rho-c_0$, $c_0\in\RR$, so that $\tilde\rho$ has zero mean in $\Omega$. Using \eqref{diffEst}, \eqref{aux1306}, and $(p^e,1)_\Omega=(\ph^*,1)_{\Omega_h}=0$, one finds
\begin{equation}\label{aux1288}
    |c_0|\lesssim h^{k}(\|p\|_{H^1(\Omega)}+\|\ph^*\|_{L^1(\Omega_h)})\le h^{k}\|p\|_{H^1(\Omega)}. 
\end{equation}

Let $\phi\in H^1(\Omega)$ be the solution to $-\Delta \phi = \tilde\rho$ in $\Omega$ with homogeneous Neumann boundary conditions, and note that by $H^2$-regularity, we have  $\|\phi\|_{H^2(\Omega)}\lesssim \|\tilde\rho\|_{\Omega}$. We extend this function smoothly to  $\phi^e:\bbR^d\to \bbR$, and let $\phi_h^*\in Q_h^*$ be the Lagrange interpolant to $\phi^e$. 

We then have
\begin{multline}\label{aux1298}
    \|\tilde\rho\|_{\Omega}^2 
= (\nab \tilde\rho, \nab \phi)_{\Omega} = (\nab \rho, \nab \phi)_{\Omega}\\
= [(\nab \rho, \nab \phi)_{\Omega} - a_{p,h}(p^e-\ph^*, \phi^e)] + a_{p,h}(p^e-\ph^*, \phi^e-\phi_h^*)
+ a_{p,h}(p^e-\ph^*, \phi_h^*).
\end{multline}

The  first two terms in the right-hand side of \eqref{aux1298} are estimated using \eqref{diffEst}, interpolation properties
(cf.~proof of Theorem~8.1 in \cite{gross2015trace}), and \eqref{pH1}:
\begin{equation}
[(\nab \rho, \nab \phi)_{\Omega} - a_{p,h}(p^e-\ph^*, \phi^e)] + a_{p,h}(p^e-\ph^*, \phi^e-\phi_h^*)\lesssim 
h \|p^e-p_h^*\|_{1,h}\|\tilde\rho\|_{\Omega}\lesssim h^{k}\|\tilde \rho\|_{\Omega}.
\end{equation}
For the last term, we use \eqref{orthog}, $|i_h(p^e,\phi_h^*)|\le h^{k+1}\|p\|_{H^{k+1}(\Omega)}\|\phi\|_{H^2(\Omega)}$ (see \cite[Lemma 5.5]{LO19}), and \cref{mainCor} to bound
\begin{equation}\label{aux1310}
\begin{split}
|a_{p,h}(p^e-\ph^*, \phi_h^*)| 
&\lesssim  h^{k+1}\|p\|_{H^{k+1}(\Omega)}\|\phi\|_{H^2(\Omega)}+  |(\bn_h\times (\nab \times (\bu_h-\bu^e)),\nab \phi_h^*)_{\Gamma_h}| + | c_f(\phi_h^*) | \\
&\lesssim h^{k+1} \|\phi\|_{H^2(\Omega)} 
+ \|\nab(\bu^e -\bu_h)\|_{\Gamma_h} \|\nab \phi_h^*\|_{\Gamma_h} + \|\fb-\fb^e\|_{-1,h} \|\phi_h^*\|_{1,h}\\
&\lesssim h^{k+1} \|\tilde \rho\|_{L^2(\Omega)} +h^{k-\frac12} \|\nab \phi_h^*\|_{\Gamma_h}.
\end{split}
\end{equation}
The last term in the right-hand side is bounded as
\begin{equation}
\label{nabPhih}
\begin{split}
\|\nab \phi_h^*\|_{\Gamma_h} &\lesssim 
\|\nab (\phi^e-\phi_h^*)\|_{\Gamma_h}+\|\nab \phi^e\|_{\Gamma_h}
\lesssim h^{\frac12}\|\phi^e\|_{H^2(\Omega^\Gamma_h)}+
\|\phi\|_{H^2(\Omega)}^{\frac12}\|\nabla\phi\|_{\Omega}^{\frac12}\\
&\lesssim h^{\frac12}\|\tilde\rho\|_{\Omega}+
\|\tilde\rho\|_{\Omega}^{\frac12}\|\tilde\rho\|_{H^{-1}(\Omega)}^{\frac12}
\lesssim \|\tilde\rho\|_{\Omega}.
\end{split}
\end{equation}
From \eqref{aux1298}--\eqref{nabPhih}, we get $\|\tilde\rho\|_{\Omega}\lesssim h^{k-\frac12}$. From this and \eqref{diffEst}--\eqref{aux1288}, we conclude \eqref{pressureError}.
\end{proof}

\begin{rmk}
The post-processed pressure error \eqref{pressureError} has optimal rate in the $H^1$ semi-norm, but suboptimal (by $1/2$) in the $L^2$ norm. However, numerical experiments in \Cref{sec:Numerics} show optimal rates, indicating that the $L^2$ error-estimate in \eqref{pressureError} is not sharp.
\end{rmk}

\begin{thm}\label{ThL2}
	Assume the solution $\bu$, $p$ to \eqref{Stokes}, $\fb$, and $\Gamma$ are sufficiently smooth. It holds
	\begin{equation}\label{vel:L2}
		\|\bu^e - \bu_h\|_{\Omega_h}\le Ch^{k+1},
	\end{equation}
	where $C$ depends on the problem data.
\end{thm}
\begin{proof}
The proof starts with a standard duality argument. Let $\bPsi_h:\,\Omega\to\Omega_h$ be the mapping satisfying \eqref{diffEst}, set $\eb:=\bu^e - \bu_h$, and note that $\tnrma{\eb}\lesssim h^k$ by \cref{mainCor}.

Consider the pair $(\bw, \eta)\in \bH^1_0(\Omega)\times L^2_0(\Omega)$ solving the Stokes problem:
	\begin{equation}\label{eq:dual}
		\left\{
		\begin{aligned}
			- \Delta \bw + \nabla \eta &= \mu_h\eb\circ\Psi_h &&\text{in }\Omega\\
			\nabla\cdot\bw &= 0 &&\text{in }\Omega
\end{aligned}\right.
	\end{equation}
By $H^2$-regularity and the properties in \eqref{diffEst} of $\bPsi_h$ we have $(\bw,\eta)\in \bH^2(\Omega)\times H^1(\Omega)$, and
\begin{equation}\label{eq:H2}
	\|\bw\|_{H^2(\Omega)} + \|\eta\|_{H^1(\Omega)} \lesssim \|\mu_h\eb\circ\bPsi_h\|_{\Omega}\lesssim \|\eb\|_{\Omega_h}.
\end{equation}	
We let $\bw^e$ be a smooth solenoidal extension of $\bw$, and $\eta^e$ a smooth extension of $\eta$.

Taking the $L^2$ inner product 	of the first equation in \eqref{eq:dual} with $\eb\circ\bPsi_h$, we get after integration by parts and a change of variables:
\[
\|\eb\|_{\Omega_h}^2=
a_h(\bw^e,\eb) + c_h(\eta^e,\eb) +\sum_{i=1}^6 I_i,
\]
where the six terms $\{I_i\}_{i=1}^6$ incorporate geometric errors and whose precise definitions are given below.
These terms  are of higher order, and we estimate them using the properties of the mapping $\bPsi_h$: 
\[
\begin{split}
	I_1&
    :=(\nab \bw,\nab(\eb\circ \bPsi_h))_\Omega - ( \nab \bw^e,\nab \eb)_{\Omega_h}
     - ((\mu_h^{-1}\nab \bw D\bPsi_h^T)\circ \bPsi_h^{-1} + \nab \bw^e,\nab \eb)_{\Omega_h}\\
& \lesssim (h^k \|(\nabla\bw)\circ\bPsi_h^{-1}\|_{\Omega_h} + \|(\nabla\bw)\circ\bPsi_h^{-1}-\nabla\bw^e\|_{\Omega_h})\|\nabla \eb\|_{\Omega_h}\\
& \lesssim (h^k \|\nabla\bw\|_{\Omega}  +h^{k+1} \|\bw^e\|_{H^2(\Omega_h)} )\|\nabla \eb\|_{\Omega_h} 
\lesssim h^{2k} \|\eb\|_{\Omega_h}.
\end{split}
\]
Here, we have used the algebraic identity $A:(BC) = (AC^T):B$.

Likewise, using $\nabla\cdot \eb=0$,
\begin{equation*}
\begin{split}
	 I_2 &:=-( \eta, \nab \cdot (\eb\circ \bPsi_h))_{\Omega}
= ((\mu_h^{-1} D\Psi_h^T - \mathbf{I})\eta\circ\bPsi_h^{-1},\nabla \eb )_{\Omega_h}\\
	 &\lesssim h^{k}\|\eta\circ\bPsi_h^{-1}\|_{\Omega_h} \|\nabla \eb \|_{\Omega_h}
	 \lesssim h^{k}\|\eta\|_{\Omega} \|\nabla \eb \|_{\Omega_h}
	 \lesssim h^{2k} \|\eb\|_{\Omega_h}.
\end{split}
\end{equation*}
Continuing, we have
\begin{align*}
I_3:
&=(\bn_h\cdot \nab \bw^e,\eb)_{\Gamma_h} -(\bn\cdot \nab \bw, (\eb\circ \Psi_h))_\Gamma
= (\bn_h\cdot \nab \bw^e - (\tilde \mu_h^{-1} \bn\cdot \nab \bw)\circ \bPsi_h^{-1},\eb)_{\Gamma_h},
\end{align*}
where $\tilde \mu_h = |{\rm cof}(D\Psi_h) \bn|$ is the surface Jacobian.
Using $|\tilde \mu_h - 1|\lesssim h^k$ and $|\bn_h-\bn|\lesssim h^k$, 
we find
\begin{align*}
I_3
&\lesssim \|\nab \bw^e-\nab \bw \circ \bPsi_h^{-1}\|_{\Gamma_h} \|\eb\|_{\Gamma_h}
+ h^k \|\nab \bw\|_{\Gamma} \|\eb\|_{\Gamma_h}
\lesssim h^{\frac32k+1} \|\eb\|_{\Omega_h},
\end{align*}
where we used $\|\nab \bw^e-\nab \bw \circ \bPsi_h^{-1}\|_{\Gamma_h}\lesssim h^{\frac{k}2+\frac12}\|\bw\|_{H^2(\Omega)}$ (cf.~the proof of~\cite[Lemma~7.3]{gross2015trace}), $\|\eb\|_{\Gamma_h}\lesssim h^{k+\frac12}$ and a trace inequality.

By similar arguments, along with \cref{mainCor},
\begin{align*}
I_4:=(\bw^e,\bn_h \cdot \nab \eb)_{\Gamma_h}
 = (\bw^e -\bw\circ \bPsi_h^{-1},\bn_h\cdot \nab \eb)_{\Gamma_h}\lesssim h^{k+1}\|\bw\|_{H^2(\Omega)} \|\nab \eb\|_{\Gamma_h}\lesssim h^{2k+\frac12} \|\eb\|_{\Omega_h},
\end{align*}
and
\begin{align*}
I_5 & := (\eta, (\eb\circ \bPsi_h)\cdot \bn)_{\Gamma} -(\eta^e,\bn_h\cdot \eb)_{\Gamma_h} = ((\tilde \mu_h^{-1}\eta  \bn)\circ \bPsi_h^{-1} -\eta^e \bn_h,\eb)_{\Gamma_h}\\
&\le h^{\frac12} \|(\tilde \mu_h^{-1}\eta  \bn)\circ \bPsi_h^{-1} -\eta^e \bn_h\|_{\Gamma_h} \tnrma{\eb}\lesssim h^{2k+\frac12} \|\eb\|_{\Omega_h}.
\end{align*}

Finally,
\[
\begin{split}
I_6 &:= -\gamma_nh^{-1}(\bw^e-\bw\circ\bPsi_h^{-1}),\eb)_{\Gamma_h}\lesssim
h^{-\frac12}\|\bw^e-\bw\circ\bPsi_h^{-1}\|_{\Gamma_h}\tnrma{\eb}
\lesssim h^{2k+\frac12} \|\eb\|_{\Omega_h}.
\end{split}
\]

Applying the estimates of $I_i$, 
we conclude that for arbitrary $\bw_h\in \bV_h$ and $\eta_h\in \Sigma_h$ it holds  
\begin{equation}\label{eqn:StartA}
\|\eb\|_{\Omega_h}^2\lesssim 
a_h(\bw^e-\bw_h,\eb) + c_h(\eta^e-\eta_h,\eb) + a_h(\bw_h,\eb)  +c_h(\eta_h,\eb) + h^{\frac32 k+1}\|\eb\|_{\Omega_h}.
\end{equation}
We now let $\bw_h=\pi_V\bw^e$, given by \cref{lem:Fortin}
and note that $\bw_h$ is divergence-free.
By \cref{thm:Main} and \cref{Prop1}, 
along with standard arguments, one estimates 
\begin{equation}
\label{eqn:StartS1}
\begin{split}
&a_h(\bw^e-\bw_h,\eb)  + c_h(\eta^e-\eta_h,\eb) \\
&\qquad\lesssim h (\|\bw\|_{H^2(\Omega)}+ \|\eta\|_{H^1(\Omega)}) (\tnrma{\eb}+h^{\frac12} \|\nab \eb\|_{\Gamma_h})\lesssim h^{k+1}\|\eb\|_{\Omega_h}.
\end{split}
\end{equation}

Next, we write
\begin{equation}
\label{eqn:StartS2}
\begin{split}
a_h(\bw_h,&\eb)  = a_h(\bw_h,\bu^e)-a_h(\bw_h,\uh)\\
& = a_h(\bw_h,\bu^e)-\fb_h(\bw_h) + b_h(p_h,\bw_h)+c_h(\lambda_h,\bw_h)+i_h(\bw_h,\bu_h)\\
& = \big[a_h(\bw_h,\bu^e)+(\Delta \bu^e,\bw_h)_{\Omega_h}\big] + (\fb^e-\fb,\bw_h)_{\Omega_h} +c_h(\lambda_h-p^e,\bw_h)+i_h(\bw_h,\bu_h).
\end{split}
\end{equation}
Here, we have used the fact that $\fb^e = -\Delta \bu^e+\nab \bar p^e $, $\bar p^e = p^e$ in $\bar \Omega_h$,  integration by parts, and the divergence-free property of $\bw_h$.

Integrating by parts and applying \cref{lem:Consist} with $m=2$ yields
\begin{align*}
\big|a_h(\bw_h,\bu^e) +(\Delta \bu^e,\bw_h)_{\Omega_h}\big|
&\lesssim \left| -(\bn_h\cdot \nab \bw_h,\bu^e)_{\Gamma_h}+{\gamma_n} h^{-1} (\bu^e,\bw_h)_{\Gamma_h}\right|
+ h^{k+1}  \|\bw_h\|_{H^2(\Omega_h^{\calT})}\\
&\lesssim |(\bn_h\cdot \nab \bw_h,\bu^e)_{\Gamma_h}|+{h}^{-1}|(\bu^e,\bw_h)_{\Gamma_h}|
+ h^{k+1}  \|\eb\|_{\Omega_h}.
\end{align*}
We then estimate
\begin{align*}
h^{-1}(\bu^e,\bw_h)_{\Gamma_h} 
&= h^{-1} (\bu^e-\bu\circ \bPsi_h^{-1},\bw^e-\bw\circ \bPsi_h^{-1})_{\Gamma_h}
+h^{-1}(\bu^e-\bu\circ \bPsi_h^{-1},\bw_h-\bw^e)_{\Gamma_h}\\
&\lesssim (h^{2k+1}+h^{k+\frac32}) \|\bu\|_{H^2(\Omega)} \|\bw\|_{H^2(\Omega)}\lesssim h^{k+\frac32} \|\eb\|_{\Omega_h}.
\end{align*}
The same reasoning shows $\|\bu^e\|_{\Gamma_h}\lesssim h^{k+1}$, and hence, we get
\[
\begin{split}
|(\bn_h\cdot \nab \bw_h,\bu^e)_{\Gamma_h}| &\lesssim
h^{k+1}\|\nabla \bw_h\|_{\Gamma_h}
\lesssim
h^{k+\frac12}\|\nabla \bw_h\|_{\Omega^\Gamma_h}
\lesssim
h^{k+\frac12}(\|\nabla (\bw_h-\bw^e)\|_{\Omega^\Gamma_h} +  \|\nabla \bw^e\|_{\Omega^\Gamma_h})\\
	&\lesssim
h^{k+1}\|\bw\|_{H^2(\Omega)}
\lesssim 
h^{k+1}\|\eb\|_{\Omega_h},
\end{split}
\]
where we used \cite[Lemma 4.10]{elliott2013finite} in the second-to-last inequality. Combining the previous three estimates, we obtain
\begin{equation}\label{eqn:StartS2A}
\big|a_h(\bw_h,\bu^e) +(\Delta \bu^e,\bw_h)_{\Omega_h}\big|
\lesssim h^{k+1}\|\eb\|_{\Omega_h}.
\end{equation}

To estimate $c_h(\bw_h,\lambda_h-p^e)$, we use that $\|\lambda_h-p^e\|_{L^2(\Gamma_h)}\le C h^k$, which follows 
from the triangle inequality,
$\|\lambda_h-p^e\|_{L^2(\Gamma_h)}\le 
\|\lambda_h-\pi_\Sigma \lambda^e\|_{L^2(\Gamma_h)}+\|\pi_\Sigma \lambda^e-\lambda^e\|_{L^2(\Gamma_h)}+\|\lambda^e-p^e\|_{L^2(\Gamma_h)},
$
and then an application of  the convergence estimate \eqref{eqn:ThmEstimate}, interpolation and geometric error estimates as in \eqref{aux1140}--\eqref{aux1145}.
We then use $\bw=0$ on $\Gamma$ to obtain 
\begin{equation}
\label{eqn:StartS2B}
\begin{split}
c_h(\lambda_h-p^e,\bw_h)
&\lesssim  h^k\|\bw_h\|_{\Gamma_h} 
\lesssim  h^k(\|\bw_h-\bw^e\|_{\Gamma_h}+
\|\bw^e-\bw\circ\bPsi_h^{-1}\|_{\Gamma_h})\\ 
&\lesssim h^{k+\frac32}\|\bw\|_{H^2(\Omega)} \lesssim  h^{k+\frac32} \|\eb\|_{\Omega_h}.
\end{split}
\end{equation}
Due to \cref{Prop1}, we have
\begin{equation}\label{eqn:StartS2C}
(\fb^e-\fb,\bw_h)_{\Omega_h}\lesssim h^{2k+\frac52} \tnrma{\bw_h} \lesssim h^{2k+\frac52} \|\bw\|_{H^2(\Omega)} 
\lesssim h^{2k+\frac52} \|\eb\|_{\Omega_h}.
\end{equation}
Finally, applying estimates in \eqref{aux705}--\eqref{aux707}, we get
\begin{equation}
\label{eqn:StartS2D}
\begin{split}
i_h(\bw_h,\bu_h)&= [i_h(\bw_h-\bw^e,\bu_h-\bu^e)+i_h(\bw,\bu_h-\bu^e)] + [i_h(\bw_h-\bw^e,\bu^e)+i_h(\bw^e,\bu^e)]\\
&\lesssim h \|\bw\|_{H^2(\Omega)}(\tnrma{\eb}+ h^k)\lesssim h^{k+1} \|\eb\|_{\Omega_h}. \end{split}
\end{equation}
We then apply the estimates \eqref{eqn:StartS2A}--\eqref{eqn:StartS2D} to \eqref{eqn:StartS2} to conclude
\begin{align}\label{eqn:StartS2End}
a_h(\bw_h,\eb) 
&\lesssim h^{k+1}\|\eb\|_{\Omega_h}.
\end{align}

Also using 	$c_h(\eta_h,\bu_h)=0$ and $\|\bu^e\|_{L^{\infty}(\Gamma_h)}\lesssim h^{k+1}$ again, we estimate
\begin{equation}\label{eqn:StartS3}
c_h(\eta_h,\eb) =  c_h(\eta_h,\bu^e) \lesssim h^{k+1}\|\eta_h\|_{\Gamma_h}\lesssim 
h^{k+1}\|\eta\|_{H^1(\Omega)}
\lesssim 
h^{k+1}\|\eb\|_{\Omega_h}.
\end{equation}
The estimate for $\|\eb\|_{\Omega_h}$ in \eqref{vel:L2} now follows by applying the estimates \eqref{eqn:StartS1}, \eqref{eqn:StartS2End}, and \eqref{eqn:StartS3} to \eqref{eqn:StartA}.
\end{proof}

\section{Numerical Examples}\label{sec:Numerics}

The method is implemented using \texttt{Netgen}/\texttt{NGSolve}\footnote{See also \url{ngsolve.org}}~\cite{Sch97,Sch14} with an add-on for the implementation of the velocity space and \texttt{ngsxfem}~\cite{LHPvW21}, an add-on to NGSolve for unfitted finite element discretizations. The experiments are fully reproducible; see the Data Availability statement below.

\subsubsection*{Implementation}
For $k=2$, we implemented the space manually in a small C++ add-on module to \texttt{NGSolve}. An alternative approach is to compute the basis automatically at runtime using the Conforming Trefftz framework presented in \cite{jcmeyermaster}. This is an extension of the Embedded Trefftz method \cite{LS23}. The idea is to start with a discontinuous Galerkin (DG) space and then compute an embedding to generate the desired subspace. In the Embedded Trefftz method, the element-wise kernel of a given differential operator is extracted to create a Trefftz DG subspace, while the Conforming Trefftz setting enforces conformity conditions on the boundary of each element to generate a desired finite element space. In this setting, we start with a Piola-mapped vector-valued DG space and enforce the continuity at the vertices and Gauss--Lobatto nodes through the Trefftz embedding. For \texttt{NGSolve}, this functionality is implemented in the add-on package \texttt{NGSTrefftz} \cite{Sto22}.

\subsection{Example 1: Stability} 
As a first test, we investigate how small intersections between the mesh and the geometry, as well as mesh deformation, affect the eigenvalues closest to zero and the condition number of the resulting system matrix.

\subsubsection*{Setup}

We consider the background domain $\Oext = (0,1)^2$. Similar to \cite{dPLM17}, we consider a rounded square $\Omega = \{(x,y)^T\in\RR^2\mid \phi(x,y) \coloneqq ((x - 0.5)^{12} + (y - 0.5)^{12})^{1/12} - (r + \varepsilon) < 0\}$. The macro mesh is a symmetric and structured crossed mesh. These elements are then split barycentrically. The resulting mesh is shown in \Cref{fig:ex1_mesh}. By choosing $r$ appropriately and varying $\varepsilon>0$, we can control the size of the small cuts between the mesh and geometry; see also the right of \Cref{fig:ex1_mesh}.

\begin{figure}
  \centering
  \includegraphics[width=6cm]{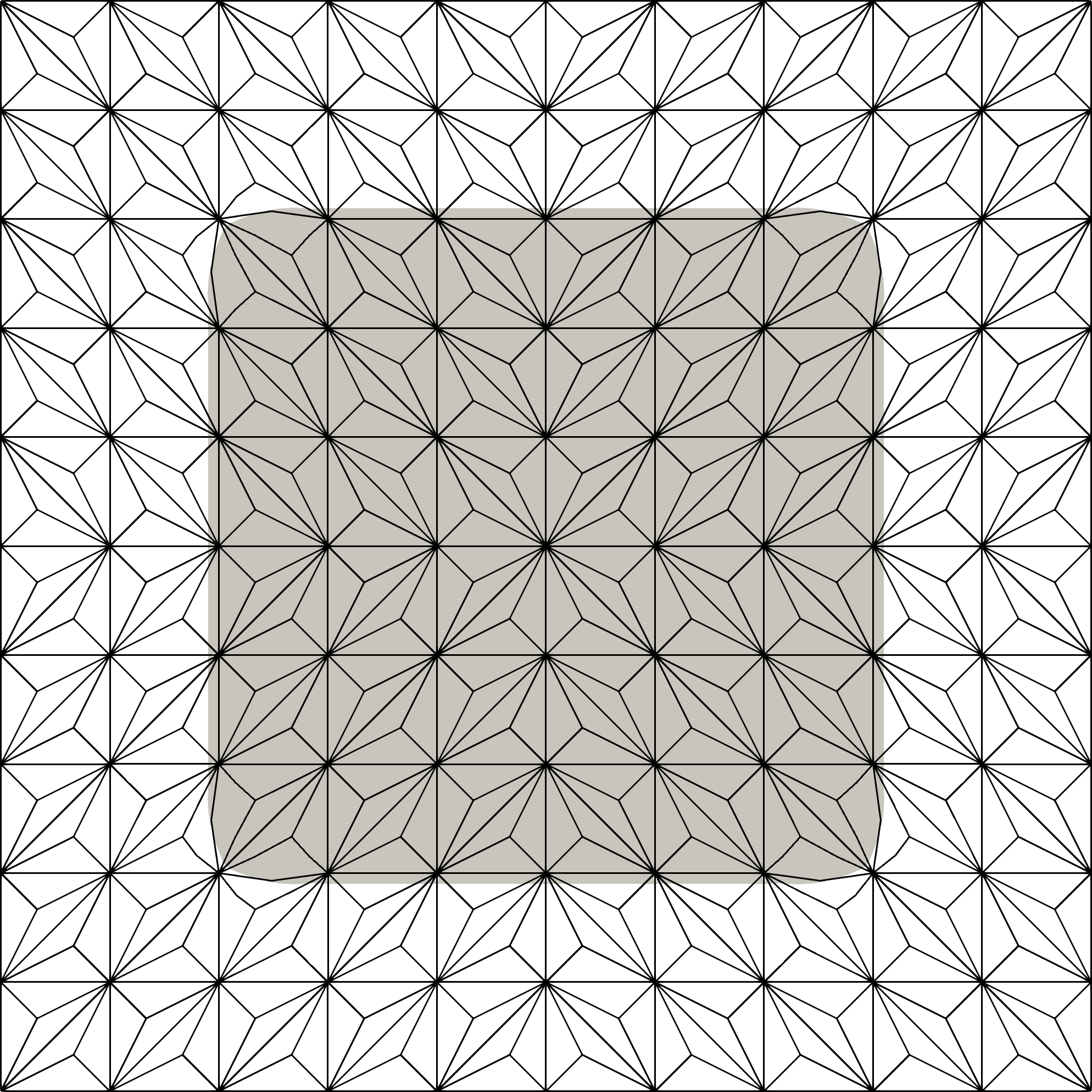}
  \hspace{1cm}
  \includegraphics{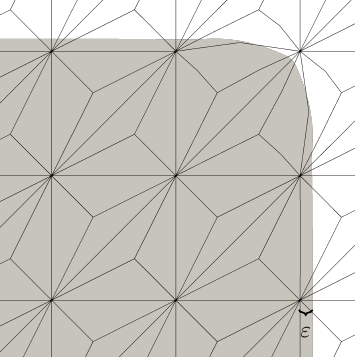}
  \caption{Mesh and Geometry used in Example 1 for the Small Cut Problem.}
  \label{fig:ex1_mesh}
\end{figure}

\subsubsection*{Results}
The resulting system matrix is indefinite. To verify that the stabilization controls the small-cut problem, we track the largest negative eigenvalue, $\lambda_\textup{max}^-$, and the smallest positive eigenvalue, $\lambda_\textup{min}^+$, of the system matrix as functions of the minimum cut ratio. We perform these computations on a fixed mesh obtained from 10 horizontal and 10 vertical subdivisions of the unit square, divided into triangles. Note that while the mesh stays fixed in these experiments, the changing level set function leads to a changing deformation of the mesh. 

We take $k=2$ and $k_\lambda=1$, $\gamma_\lambda=0.1$, $\gamma_n=40$ and $\gamma_\textup{gp}\in\{0.01, 0.1, 1\}$. The results are shown in \Cref{fig:example1.eigenvals-gp-study}. We can see that both eigenvalues vary significantly. However, the dependence is not proportional to the smallest cut $\min \abs{K\cap\Oh}/\abs{K}$. While the smallest cut varies over ten orders of magnitude, the largest variation in the eigenvalues, realized by the smallest positive eigenvalue with the smallest stabilization parameter, is four orders of magnitude. For larger ghost-penalty parameters, this reduces to one order of magnitude. We also see that an increasing ghost-penalty parameter helps to bound the smallest positive eigenvalue further away from zero. However, the largest negative eigenvalue becomes smaller in magnitude. Nevertheless, this negative effect is not as large as the positive effect on the smallest positive eigenvalue, leading to a net positive effect. Furthermore, we note that the changing mesh deformation also has an influence on the eigenvalues.

\begin{figure}
  \centering
  \includegraphics{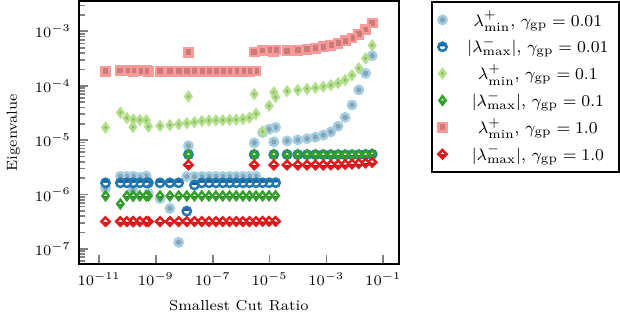}
  \caption{Smallest positive and largest negative eigenvalues of the system matrix for varying cut configurations and ghost-penalty parameters in Example 1.}
  \label{fig:example1.eigenvals-gp-study}
\end{figure}

Furthermore, we also consider the condition number $\kappa=\abs{\lambda_\textup{max}}/\abs{\lambda_\textup{min}}$ of the system matrix under mesh refinement. We consider a series of six meshes where the unit square is divided both horizontally and vertically into $8 \cdot 2^{i}$, $i=0,\ldots,5$ elements, each of which is divided into triangles and then split barycentrically. The results are shown in Table~\ref{tab.condition}, where the rate is with respect to the mesh parameter. Again, we see that the condition number varies, but its dependence is not proportional to the cut ratio. In fact, it appears to depend on the size of the deformation, which decreases with mesh refinement. After several refinements, the quadratic scaling of the condition number with respect to mesh size becomes the more significant factor.

\begin{table}
  \caption{Condition number estimates over a series of meshes in Example 1}
  \label{tab.condition}
  \begin{tabular}{lllllr}
    \toprule
    lvl & Cut Ratio & $\abs{\lambda_\textup{min}}$ & $\abs{\lambda_\textup{max}}$ & $\kappa$ & Rate\\
    \midrule
    0 & $1.34\cdot10^{-1}$ & $2.74\cdot10^{-6}$ & $8.27\cdot10^{ 3}$ & $3.01\cdot10^{ 9}$ & ---\\
    1 & $1.37\cdot10^{-1}$ & $1.51\cdot10^{-6}$ & $1.21\cdot10^{ 3}$ & $8.01\cdot10^{ 8}$ & $-1.9$\\
    2 & $7.35\cdot10^{-3}$ & $5.21\cdot10^{-7}$ & $6.47\cdot10^{ 1}$ & $1.24\cdot10^{ 8}$ & $-2.7$\\
    3 & $1.35\cdot10^{-4}$ & $1.29\cdot10^{-7}$ & $1.02\cdot10^{ 2}$ & $7.97\cdot10^{ 8}$ & $2.7$\\
    4 & $5.44\cdot10^{-4}$ & $3.21\cdot10^{-8}$ & $9.29\cdot10^{ 1}$ & $2.89\cdot10^{ 9}$ & $1.9$\\
    5 & $4.24\cdot10^{-4}$ & $7.98\cdot10^{-9}$ & $9.66\cdot10^{ 1}$ & $1.21\cdot10^{10}$ & $2.1$\\
    \bottomrule
  \end{tabular}
\end{table}

\subsection{Example 2: Convergence}
As a second example, we consider the convergence of the method for an example on a simple but smooth domain.

\subsubsection*{Setup}
We consider the background domain $\Oext = (-1,1)^2$. The domain is given by $\Omega = \{(x,y)^T\in\RR^2\mid \phi(x,y) \coloneqq (x^4 + y^4){^{1/4} - \nicefrac{1}{\sqrt{2}}} < 0\}$. For the exact solution, we choose
\begin{equation*}
  \ub_\text{ex} = \begin{pmatrix}
    4y^3 \cos(2 \pi (x^4 + y^4) )\\ -4x^3 \cos(2 \pi (x^4 + y^4))
  \end{pmatrix},
  \quad
\quad
  p_\text{ex} = \sin(\pi x y) + x^3 + y^3,
\end{equation*}
and the right-hand side is computed such that $\fb = -\nu\Delta\ub_\text{ex} + \nabla p_\text{ex}$ with $\nu=1$.

\subsubsection*{Results}

We take the initial macro mesh with {$h_0 = 0.4$}, perform five uniform refinements of this macro mesh, and compute the solution on the Alfeld split of each refined macro mesh. The Nitsche parameter is set to $\gamma_n = 40$, and the velocity ghost-penalty and Lagrange multiplier stabilization parameters are $\gamma_\textup{gp} = \gamma_{\lambda} = 0.1$.

We run our experiments with $k=2$, i.e., the $\PP_2\!-\!\PP_1^{\rm disc}$ Scott--Vogelius element, and set $k_\lambda = k-1$. By default, our implementation uses a higher-order geometry approximation satisfying \eqref{GeomError}, as suggested by the theory. We also experimented with a polygonal approximation of $\Gamma$ obtained as the zero-level set of the $P_1$ interpolant of the distance function.

The resulting errors in the velocity, velocity gradient, divergence, and pressure are shown in \Cref{fig:example2.convergence}. The discrete velocity and divergence on the coarsest mesh are displayed in \Cref{fig:example2.sol}. We observe that the velocity error converges at the optimal rates in both the $H^1$ and $L^2$ norms when the geometry approximation satisfies \eqref{GeomError}. We note that the post-processed pressure also exhibits the optimal $O(h^{k})$ convergence.
If the approximation of $\Omega$ is relaxed to the polygonal one, the velocity convergence rate deteriorates to $O(h^2)$ in the $L^2$ norm and to $O(h^{3/2})$ in the $H^1$ norm. The pressure convergence appears to degrade as well.
Finally, in agreement with the theory, the recovered velocity is divergence-free on all meshes.

\begin{figure}
  \centering
  \includegraphics{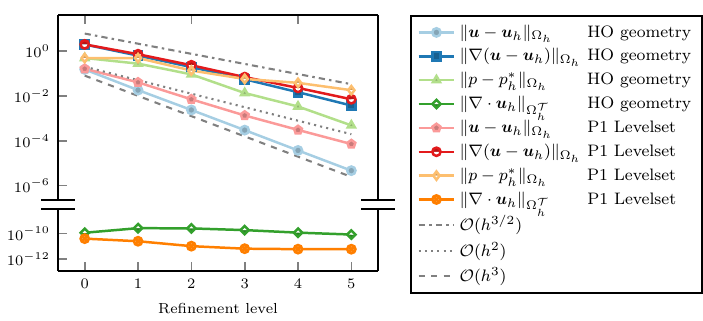}
  \caption{Convergence results for Example 2 {with $k=2$ and $k_\lambda=1$}.}
  \label{fig:example2.convergence}
\end{figure}

\begin{figure}
  \centering
  \includegraphics[width=7cm]{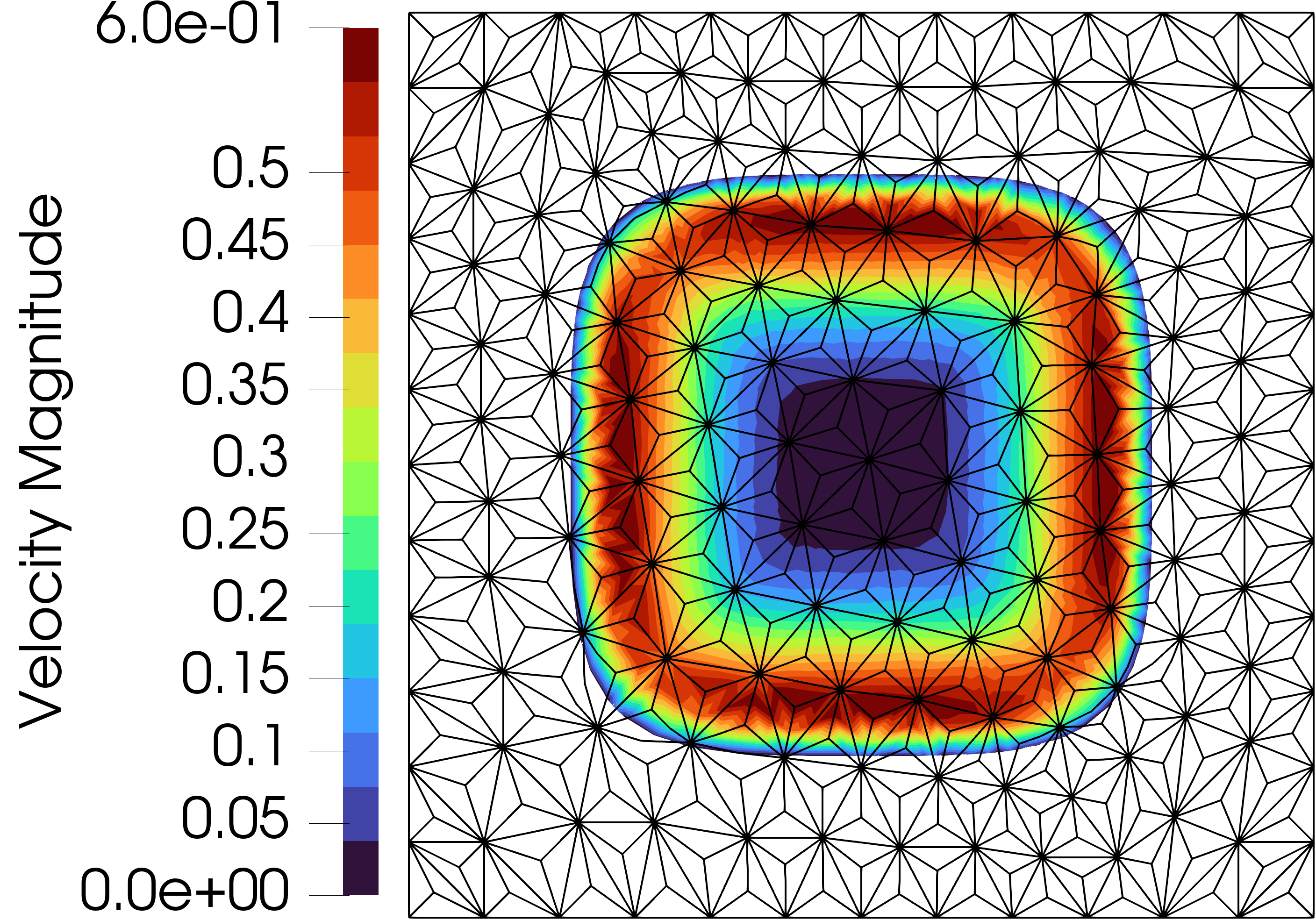}
  \includegraphics[width=7cm]{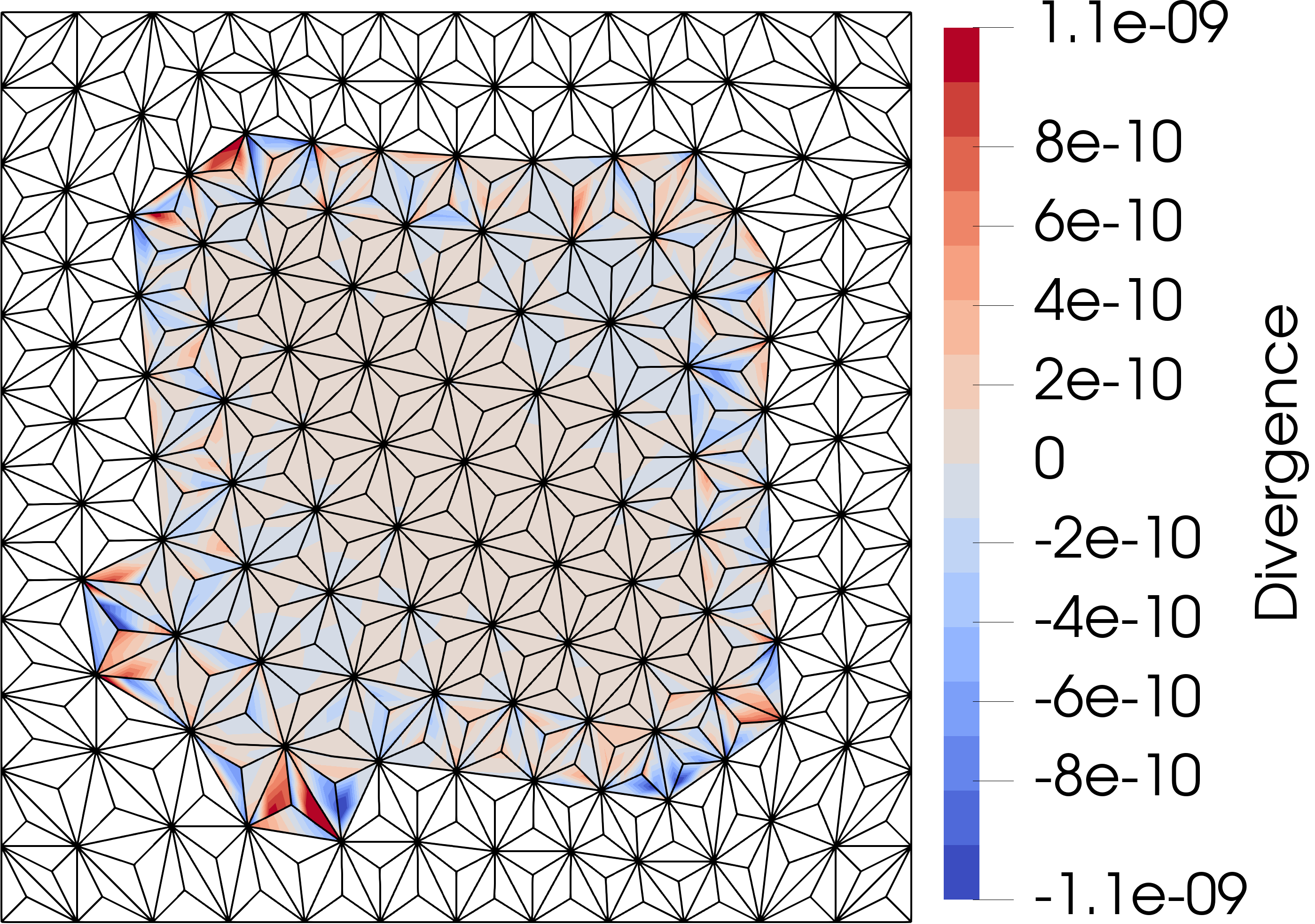}
  \caption{Velocity solution on the level set domain (left) and the discrete velocity divergence (right), both on the curved mesh.}
  \label{fig:example2.sol}
\end{figure}

We also run the experiment for $k\in\{3,4\}$ and $k_\lambda = k-1$ using the automatic runtime generation of the finite element space using the Conforming Trefftz framework. The results can be seen in \Cref{fig.example2:conformingTrefftz}. We observe optimal-order convergence for all unknowns. Note that for $k=4$, we did not compute the last refinement level due to the limitations of the direct solver used in the experiments.

\begin{figure}
  \centering
  \includegraphics{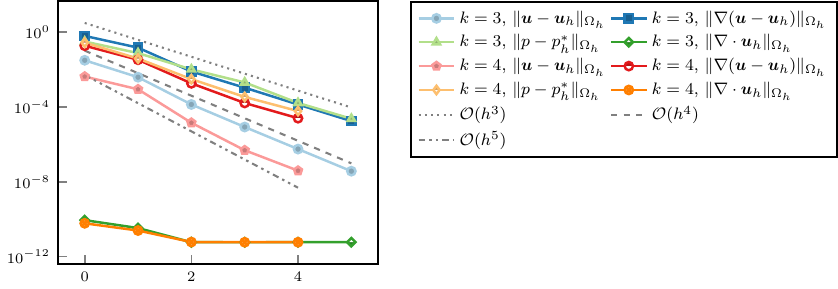}
  \caption{Convergence results for Example 2 using the Conforming Trefftz implementation of the velocity space.}
  \label{fig.example2:conformingTrefftz}
\end{figure}

\subsection{Example 3: Pressure Robustness}\label{sec:noFlow}
As a third example, we consider a no-flow problem to illustrate the pressure dependence of the velocity error in \cref{mainCor}, see also Remark~\ref{rmk.pressure-robustness}.

\subsubsection*{Setup}
We consider a slightly more complex geometry
\begin{equation*}
  \Omega \coloneqq \{(r, \theta)\in\RR_+\times[-\pi, \pi)\mid r - 0.7 + 0.2\cos(6 \theta) < 0 \}
\end{equation*}
of a smoothed six-pointed star. We consider a no-flow problem with data $\fb = 5(x^4, y^4)^T$ resulting in the exact solution $\ub=(0, 0)^T$ and $p= x^5+y^5$. This pressure has mean zero on the exact geometry.

\subsubsection*{Results}
The initial macro mesh is constructed with $h_0 = 0.3$, and we again consider a sequence of uniform mesh refinements. Each mesh is barycentrically refined to define the finite element spaces. In this experiment, we use only the higher-order approximation of the geometry but vary the polynomial degree of the Lagrange multiplier, taking $k_\lambda \in \{k-1, k\}$ {with $k=2$}. Both choices of $k_\lambda$ are admissible according to our analysis.

The convergence results are shown in \Cref{fig:example3.convergence}. We observe that the velocity finite element solution is divergence-free, but the velocity error, while small, is not zero. Moreover, choosing $k_\lambda = k$ yields a smaller velocity error with a faster decay rate compared to $k_\lambda = k-1$. In particular, we observe $\tnrma{\bu_h} = O(h^{k_\lambda+1})$, which agrees with the estimate  \eqref{eqn:ThmEstimate} for $\bu \equiv 0$. This behavior is expected, since employing a higher-order polynomial space for the Lagrange multiplier enforces the condition $\bu_h \cdot \bn_h = 0$ on $\Gamma_h$ more strongly, which in turn leads to better decoupling of the velocity and pressure errors.

\begin{figure}
  \centering
  \includegraphics{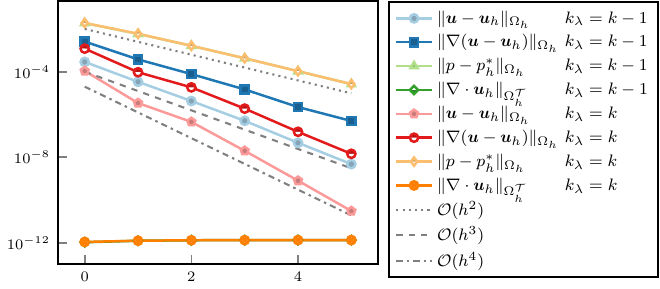}
  \caption{Convergence results for the no-flow case Example 3.}
  \label{fig:example3.convergence}
\end{figure}

\section*{Data Availability Statement}
The numerical results are reproducible using the code available on GitHub under \url{https://github.com/hvonwah/repro-unfitted-div-free-stokes}. The code and results are also archived on Zenodo \url{https://doi.org/10.5281/zenodo.21460130}.

\section*{Declaration of Generative AI Use}

Claude and ChatGPT were used after the final revision to correct typographical and grammatical mistakes.

\appendix

\section{Proof of \cref{Th:fitted}} \label{sec:A}

The mapping $\bPhi_h$ satisfies for $0\le m\le k+1$ (cf.~\cite[(3.17), (3.9) and Theorem 3.11]{lehrenfeld2018analysis})
\begin{align}\label{PhiBounds}
|\bPhi_h|_{W^{m,\infty}(\tilde K)}\lesssim 1,\qquad |\bPhi^{-1}_h|_{W^{m,\infty}(K)}\lesssim 1
\qquad \forall \tilde K\in \Th,\ \forall K\in \calT_h^{\textup{Al}}.
\end{align}
In addition, by \cite[(3.38)]{lehrenfeld2018analysis}, there holds
\begin{equation}\label{eqn:detPhi}
\|D\bPhi_h-{\bm I}\|_{L^\infty(\Omega_h^{\tilde \calT})}\lesssim h.
\end{equation}

By the shape-regularity and quasi-uniformity of $\Th$,
there holds $\det(D\varphi_{\tilde K}) = d! |\tilde K| \approx h^d$, $|\varphi_{\tilde K}|_{W^{m,\infty}(\hat T)}\lesssim h^m$,
and $|\varphi^{-1}_{\tilde K}|_{W^{m,\infty}(\tilde K)}\lesssim h^{-m}$
for $m=0,1$ and $\tilde K\in \Th$. Consequently by \eqref{PhiBounds}, we have for all $K\in \calT_h^{\textup{Al}}$
\begin{alignat}{3}\label{eqn:varphiKBounds}
&\det(D\varphi_K)\approx h^d,\qquad&&|\varphi_K|_{W^{m,\infty}(\hat T)}\lesssim h^m,\qquad &&|\varphi_K^{-1}|_{W^{m,\infty}(K)}\lesssim h^{-1}\ \ (1\le m\le (k+1)).
\end{alignat}

As a first step to prove \cref{Th:fitted},  
we provide an equivalent definition of  $\bV_h$, stating it as a macro-like finite element space.
For an affine element $\tilde T\in \Thm$ in the macro triangulation,
let $\tilde T^{\textup{Al}} = \{\tilde K_1,\ldots,\tilde K_{d+1}\}$ be 
the local Alfeld split. 
Let $\varphi_{\tilde T}:\hat T\to \tilde T$
be an affine mapping, and set $\varphi_T = \bPhi_h\circ \varphi_{\tilde T}:\hat T\to T$,
where $T = \bPhi_h(\tilde T)$; see \Cref{fig:Mappings1}.
Note that $\varphi_T$ is a continuous piecewise polynomial
with respect to the reference Alfeld split $\hat T^{\textup{Al}} = \{\hat K_1,\ldots,\hat K_{d+1}\}$, but it is not $C^1$ because $\bPhi_h$ is not 
in general $C^1$ on $\tilde T$. We see that $\varphi_T$
inherits the scaling found in \eqref{eqn:varphiKBounds}, i.e.,
\begin{alignat}{3}\label{eqn:varphiTBounds}
&\det(D\varphi_T)\approx h^d,\qquad&&|\varphi_T|_{W^{m,\infty}(\hat T)}\lesssim h^m,\qquad &&|\varphi_T^{-1}|_{W^{m,\infty}(T)}\lesssim h^{-1}\ \ (1\le m\le (k+1)),
\end{alignat}
where the norms are understood to be piecewise defined with respect to the local splits.

We then write an equivalent definition 
of the velocity space $\bV_h$, now with respect to
the piecewise diffeomorphisms $\varphi_T$:
\begin{align*}
\bV_h
& = \left\{\bv_h\in \bL^2(\Omega_h^{\calT})\middle\vert \begin{array}{c} \bv_h|_T = (\frac1{\det(D\varphi_T)}D\varphi_T \hat \bv_h)\circ \varphi_T^{-1},\ \hat \bv_h\in [\mathbb{P}^k_{\rm dc}(\hat T^{\textup{Al}})]^2\ \ \forall T\in \calT_h\\ \text{ and $\bv_h$ is
single-valued at Lagrange nodes in $\calT_h^{\textup{Al}}$}\end{array}\right\},
\end{align*}
where $\mathbb{P}^k_{\rm dc}(\hat T^{\textup{Al}})$ is the 
space of piecewise polynomials of degree $\le k$ with respect
to the reference Alfeld split $\hat T^{\textup{Al}}$. 
Likewise, the pressure
space is equivalently written as
\begin{align*}
Q_h = \{q_h\in L^2(\Omega^{\calT}_h):\ q_h|_T\circ \varphi_T\in \mathbb{P}^{k-1}_{\rm dc}(\hat T^{\textup{Al}})\ \forall T\in \calT_h\}.
\end{align*}

\begin{figure}[ht]
\centering
\includegraphics{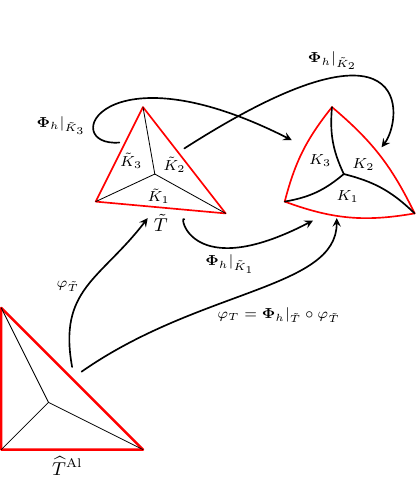}
\caption{\label{fig:Mappings1} A pictorial description
of the construction of $\varphi_T$. Note
that $\bPhi_h$ is continuous piecewise polynomial,
but is not $C^1$ on $\tilde T$.}
\end{figure}

\begin{figure}[ht]
\centering
\includegraphics{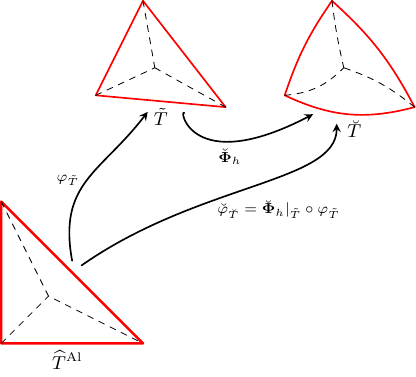}
\caption{\label{fig:Mappings2}A pictorial description of the mappoing $\breve \varphi_{\breve T}$. The mapping
$\breve\bPhi_h$ is a polynomial on $\tilde T$ and hence smooth.}
\end{figure}

Next, we define an auxiliary space closely related to the one given in \cite{NO21,DN24}
that we will use in the proof of \cref{Th:fitted}. We set
\begin{align*}
\breve{\mathcal{T}}_h = \{\breve{T} = \breve{\bPhi}_h(\tilde T):\   \tilde T \in \Thm\},
\end{align*}
where $\breve{\bPhi}_h$ is the piecewise polynomial mapping
with respect to $\Thm$, as constructed in \cite{Leh16}.
Note that $\bPhi_h$ and $\breve{\bPhi}_h$ are both piecewise polynomials,
but $\bPhi_h$ is defined with respect to $\Th$,
whereas $\breve{\bPhi}_h$ is defined with respect to $\Thm$.
Therefore, we have $\bPhi_h(\tilde T)\neq \breve\bPhi_h(\tilde T)$ ($\tilde T\in \Thm$) in general. 

For each $\breve{T}\in \breve{\calT}_h$, let $\breve{\varphi}_{\breve T}:\hat T\to \breve{T}$
be given by $\breve{\varphi}_{\breve T} = \breve{\bPhi}_h\circ \varphi_{\tilde T}$,
where $\varphi_{\tilde T}:\hat T\to \tilde T\in \Thm$ is the same
affine diffeomorphism as above (where $\breve T = \breve{\bPhi}_h(\tilde T)$); 
see \Cref{fig:Mappings2}.
{The mapping $\breve\bPhi_h$ satisfies similar bounds
as in \eqref{PhiBounds} due to the results in \cite{lehrenfeld2018analysis}}, i.e.,
\begin{align}\label{brevePhiBounds}
|\breve \bPhi_h|_{W^{m,\infty}(\tilde T)}\lesssim 1,\qquad |\breve \bPhi^{-1}_h|_{W^{m,\infty}(\breve T)}
\qquad \forall \tilde T\in \Thm,\ \forall \breve T\in \breve\calT_h,
\end{align}
as well as
\begin{equation}\label{eqn:detBrevePhi}
\|D\breve \bPhi_h - {\bm I}\|_{L^\infty(\Omega_h^{\tilde \calT})}\lesssim h.
\end{equation}
Then the mapping $\breve \varphi_{\breve T}$ satisfies
estimates analogous to \eqref{eqn:varphiTBounds}, i.e.,
\begin{alignat}{3}\label{eqn:brevevarphiTBounds}
 &\det(D\breve\varphi_{\breve T})\approx h^d,\qquad&&|\breve\varphi_{\breve T}|_{W^{m,\infty}(\hat T)}\lesssim h^m,\qquad &&|\breve\varphi_{\breve T}^{-1}|_{W^{m,\infty}(\breve T)}\lesssim h^{-1}\ \ (1\le m\le (k+1)).
\end{alignat}

Let $\bbP_{\rm c}^k(\hat T^{\textup{Al}}) = \bbP_{\rm dc}^k(\hat T^{\textup{Al}})\cap H^1(\hat T)$ be the $k$th-degree local Lagrange space
with respect to the local Alfeld split.
Set $\breve{\Omega}_h^{\calT}$ to be the associated
domain of the mesh $\breve \calT_h$ and define
\begin{align*}
\breve{\bV}_h = \left\{\breve{\bv}_h\in \bL^2(\breve{\Omega}^{\calT}_h)\middle\vert \begin{array}{c}\ \breve{\bv}_h|_{\breve T} = (\frac1{\det(D\breve{\varphi}_{\breve T})} D\breve{\varphi}_{\breve T} \hat{\breve{\bv}}_h)\circ \breve{\varphi}_{\breve T}^{-1}\ \exists \hat{\breve{\bv}}_h\in [\mathbb{P}^k_{\rm c}(\hat T^{\textup{Al}})]^2\ \ \forall \breve{T}\in \breve\calT_h\\ \text{$\breve{\bv}_h$ is single-valued at Lagrange nodes in }\breve{\calT}_h\end{array}\right\},
\end{align*}
where the Lagrange nodes of $\breve{\calT}_h$ are the $\breve\bPhi_h$-mapped
Lagrange nodes of $\Thm$.

Set $\breve{\bV}_h^0$ to be the subspace of functions in $\breve{\bV}_h$ that
vanish on $\p \breve{\Omega}^{\calT}_h$, and
define the discrete pressure spaces with respect to the macro mesh:
\begin{align*}
\breve{Q}_h = \{\breve{q}_h\in L^2(\breve{\Omega}^{\calT}_h)|\ \breve{q}_h\circ \varphi_{\breve{T}}\in \bbP_{\rm dc}^{k-1}(\hat T^{\textup{Al}})\},\qquad \breve{Q}_h^0 = \breve{Q}_h\cap  L^2_0(\breve{\Omega}^{\calT}_h).
\end{align*}

\begin{lem}\label{lem:BreveStable}
{There holds $\breve{\bV}^0_h\subset {\bm H}({\rm div};\breve\Omega_h^{\calT})$.}
Moreover, the pair $\breve{\bV}^0_h\times \breve{Q}^0_h$ is inf-sup stable for $k\ge d$: there exists $\breve{\beta}>0$
independent of the size of the mesh such that
\begin{align*}
\breve{\beta} \|\breve{q}_h\|_{\breve\Omega^{\calT}_h} \le \sup_{\breve{\bv}_h\in \breve{\bV}^0_h\backslash \{0\}} \frac{\int_{\Omega_h^{\calT}} ({\rm div}\,\breve{\bv})\breve{q}_h}{\|\breve{\bv}_h\|_{H^1(\breve\Omega^{\calT}_h)}}\qquad \forall \breve{q}_h\in \breve{Q}^0_h.
\end{align*}
\end{lem}
\begin{proof}
The proof directly follows from the arguments 
in \cite[Theorem 4.2]{NO21}, \cite[Theorem 4.4]{NO21}, and \cite[Theorem 4.9]{DN24}.
Therefore, its proof is omitted.
\end{proof}

We establish an analogous inf-sup stability result for $\bV^0_h\times Q^0_h$ 
based on \cref{lem:BreveStable} and a perturbation argument.
To ease presentation, we set $\breve{A}_{\breve T} =  \frac1{\det(D\breve{\varphi}_{\breve T})}D\breve{\varphi}_{\breve T}$
and $A_T = \frac1{\det(D{\varphi}_T)}D{\varphi}_T$, so that the velocity spaces are given by
\begin{align*}
\bV_h
& = \left\{\bv_h\in \bL^2(\Omega_h^{\calT})\middle\vert \begin{array}{c} \bv_h|_T = (A_T \hat \bv_h)\circ \varphi_T^{-1},\ \hat \bv_h\in [\mathbb{P}^k_{\rm dc}(\hat T^{\textup{Al}})]^2\ \ \forall T\in \calT_h\\ \text{ and $\bv_h$ is
single-valued at Lagrange nodes in $\calT_h^{\textup{Al}}$}\end{array}\right\},\\
\breve{\bV}_h 
&= \left\{\breve{\bv}_h\in \bL^2(\breve{\Omega}^{\calT}_h)\middle\vert \begin{array}{c}\ \breve{\bv}_h|_{\breve T} = (\breve A_{\breve T} \hat{\breve{\bv}}_h)\circ \breve{\varphi}_{\breve T}^{-1}\ \exists \hat{\breve{\bv}}_h\in [\mathbb{P}^k_{\rm c}(\hat T^{\textup{Al}})]^2\ \ \forall \breve{T}\in \breve\calT_h\\ \text{$\breve{\bv}_h$ is single-valued at Lagrange nodes in }\breve{\calT}_h\end{array}\right\}.
\end{align*}

From \eqref{eqn:varphiTBounds} and \eqref{eqn:brevevarphiTBounds},
we can derive (cf.~\Cref{app-ABounds})
\begin{equation}
\label{eqn:ABounds}
\begin{aligned}
&|A_T|_{W^{m,\infty}(\hat T)}\lesssim h^{m-d+1},\qquad &&|A_T^{-1}|_{W^{m,\infty}(\hat T)}\lesssim h^{m+d-1},\\
&|\breve A_{\breve T}|_{W^{m,\infty}(\hat T)}\lesssim h^{m-d+1},\qquad &&|\breve A_{\breve T}^{-1}|_{W^{m,\infty}(\hat T)}\lesssim h^{m+d-1}.
\end{aligned}
\end{equation}

\begin{lem}\label{lem:vPert}
For any $\breve{\bv}_h\in \breve{\bV}_h$, there exists
$\bv_h\in \bV_h$ such that ($m=0,1$)
\begin{align}\label{eqn:PertV}
|\bv_h\circ \bPhi_h-\breve{\bv}_h\circ \breve{\bPhi}_h|_{H^m(\tilde T)}\lesssim h^{2-m}\|\breve{\bv}_h\|_{H^1(\breve T)}\qquad \forall \tilde T\in \Thm,\qquad \breve T = \breve{\bPhi}_h(\tilde T).
\end{align}
\end{lem}
\begin{proof}
Given $\breve{\bv}_h\in \breve{\bV}_h$, we uniquely determine 
$\bv_h\in \bV_h$ such that $\bv_h(a) = \breve{\bv}_h(\breve{a})$
at all Lagrange nodes $a$ in the triangulation 
$\calT_h^{\textup{Al}}$, where $\breve{a} = (\breve{\bPhi}_h\circ \bPhi^{-1}_h)(a)$.

Fix $\tilde T\in \Thm$ and set $T = \bPhi_h(\tilde T)$
and $\breve{T} = \breve\bPhi_h(\tilde T)$.
Write $\breve{\bv}_h|_{\breve{T}} = (\breve{A}_{\breve T}  \hat{\breve \bv}_h)\circ \varphi_{\breve{T}}^{-1}$ 
and $\bv_h|_T = (A_T \hat \bv_h)\circ \varphi_T^{-1}$ for some $\hat{\breve \bv}_h\in [\mathbb{P}^k_{\rm c}(\hat T^{\textup{Al}})]^d$
and $\hat \bv_h\in [\mathbb{P}^k_{dc}(\hat T^{\textup{Al}})]^d$.
We then set $\tilde \bv_h = \bv_h\circ \bPhi_h:\tilde T\to \bbR^d$
and $\tilde{\breve{\bv}}_h = \breve{\bv}_h\circ \breve{\bPhi}_h:\tilde T\to \bbR^d$,
so that $\tilde \bv_h(\tilde a) = \tilde{\breve{\bv}}_h(\tilde a)$
at all Lagrange points $\tilde a$ in $\Th$,
and 
\begin{align*}
\tilde{\breve{\bv}}_h = (\breve{A}_{\breve T} \hat{\breve{\bv}}_h)\circ {\varphi}_{\tilde T}^{-1},\qquad
\tilde{\bv}_h = ({A}_T \hat{{\bv}}_h)\circ {\varphi}_{\tilde T}^{-1}.
\end{align*}

We have, by a change of variables,
\begin{equation}
\label{eqn:m0Start}
\begin{split}
\|\tilde \bv_h - \tilde{\breve{\bv}}_h\|_{L^2(\tilde T)}^2
&\approx h^d \int_{\hat T} \left|A_T \hat \bv_h - \breve{A}_{\breve T} \hat{\breve{\bv}}_h\right|^2\dif \hat{\xb}.
\end{split}
\end{equation}

We use the identity
\begin{align*}
\breve{A}_{\breve T}(\hat a) \hat{\breve \bv}_h(\hat a) = A_T(\hat a) \hat \bv_h(\hat a)
\qquad \text{at all Lagrange points $\hat a$ in $\hat T^{\textup{Al}}$},
\end{align*}
to conclude
\begin{align}\label{eqn:hatvhInterp}
\hat \bv_h(\hat a) = A_T^{-1}(\hat a) \breve{A}_{\breve T}(\hat a) \hat{\breve{\bv}}_h(\hat a).
\end{align}
Thus, $\hat \bv_h|_{\hat K_j}$ is the $k$-degree interpolant of  $ A_T^{-1} \breve{A}_{\breve T} \hat{\breve{\bv}}_h$ on each $\hat K\in \hat T^{\textup{Al}}$. We can then apply the Bramble-Hilbert Lemma 
and $\|A_T\|_{L^\infty(\hat T)}\lesssim h^{1-d}$ (cf.~\eqref{eqn:ABounds}) to obtain
\begin{equation}
\label{eqn:I1Step1}
\begin{split}
\|\tilde \bv_h - \tilde{\breve{\bv}}_h\|_{L^2(\tilde T)}^2&\approx h^d  \int_{\hat T} |A_T|^2 \left| \hat \bv_h - A_T^{-1}\breve{A}_{\breve T} \hat{\breve{\bv}}_h\right|^2\dif \hat{\xb}
\lesssim h^{2-d} |A_T^{-1} \breve{A}_{\breve T} \hat{\breve{\bv}}_h|^2_{H^{k+1}(\hat T)}.
\end{split}
\end{equation}

Next, we apply the product rule, \eqref{eqn:ABounds}, \eqref{eqn:brevevarphiTBounds}, 
and a change of variables to bound
the right-hand side of \eqref{eqn:I1Step1}:
\begin{equation}\label{eqn:RHSI1Step1}
\begin{split}
 |A_T^{-1} \breve{A}_{\breve T} \hat{\breve{\bv}}_h|^2_{H^{k+1}(\hat T)}
&\lesssim  \sum_{\ell=0}^{k+1} |A_T^{-1}|^2_{W^{\ell,\infty}(\hat T)} |\breve{A}_{\breve T} \hat{\breve{\bv}}_h|^2_{H^{k+1-\ell}(\hat T)}\lesssim  h^{2d-2} \sum_{\ell=0}^{k+1} h^{2\ell} |\breve{A}_{\breve T} \hat{\breve{\bv}}_h|^2_{H^{k+1-\ell}(\hat T)}\\
&\lesssim h^{2d-2} |\breve{A}_{\breve T} \hat{\breve{\bv}}_h|^2_{H^{k+1}(\hat T)}+h^{d-2}_T \sum_{\ell=1}^{k+1} h^{2\ell}
\cdot h^{2(k+1-\ell)}  \|\breve{\bv}_h\|^2_{H^{k+1-\ell}( {\breve T})}\\
&\lesssim h^{2d-2} |\breve{A}_{\breve T} \hat{\breve{\bv}}_h|^2_{H^{k+1}(\hat T)}+h^{d+2} \|\breve{\bv}_h\|_{H^1(\breve T)}^2,
\end{split}
\end{equation}
where we used an inverse estimate in the last step.
Continuing, we use that 
$\hat{\breve{\bv}}_h$ is a piecewise polynomial of degree $\le k$,
the product rule, \eqref{eqn:ABounds}, and an inverse estimate:
 \begin{align*}
  |\breve{A}_{\breve T} \hat{\breve{\bv}}_h|_{H^{k+1}(\hat T)}
  &  \le \sum_{\ell=0}^{k+1} |\breve{A}_{\breve T}|_{W^{k+1-\ell,\infty}(\hat T)} |\hat{\breve{\bv}}_h|_{H^\ell(\hat T)}
  =   \sum_{\ell=0}^{k} |\breve{A}_{\breve T}|_{W^{k+1-\ell,\infty}(\hat T)} |\hat{\breve{\bv}}_h|_{H^\ell(\hat T)}\\
  &\lesssim  \sum_{\ell=0}^{k} h^{k+2-\ell-d} |\hat{\breve{\bv}}_h|_{H^\ell(\hat T)}
  =  \sum_{\ell=0}^{k} h^{k+2-\ell-d} |\breve{A}_{\breve T}^{-1}\breve{A}_{\breve T}\hat{\breve{\bv}}_h|_{H^\ell(\hat T)}\\
  &\lesssim \sum_{\ell=0}^k h^{k+2-\ell-d} \sum_{m=0}^\ell |\breve{A}_{\breve T}^{-1}|_{W^{\ell-m,\infty}(\hat T)} |\breve{A}_{\breve T}\hat{\breve{\bv}}_h|_{H^m(\hat T)}\\
  &\lesssim   h^{k+1-d/2} \sum_{m=0}^k \|\breve{\bv}_h\|_{H^m(\breve T)}
  \lesssim h^{k+1-d/2} \|\breve{\bv}_h\|_{H^k(\breve T)}\lesssim h^{2-d/2}\|\breve{\bv}_h\|_{H^1(\breve T)}.
  \end{align*}
Applying this estimate to \eqref{eqn:RHSI1Step1} yields
  \begin{align}\label{eqn:HONorm}
  |A_T^{-1} \breve{A}_{\breve T} \hat{\breve{\bv}}_h|^2_{H^{k+1}(\hat T)} \lesssim h^{d+2} \|\breve{\bv}_h\|_{H^1(T)}^2,
  \end{align}
and so by \eqref{eqn:I1Step1},
  \begin{align}\label{eqn:I1Bound}
  \|\tilde \bv_h - \tilde{\breve{\bv}}_h\|_{L^2(\tilde T)}^2\lesssim h^4  \|\breve{\bv}_h\|_{H^1(\breve T)}^2.
  \end{align}
Thus, \eqref{eqn:PertV} holds in the case $m=0$.

Next, we prove \eqref{eqn:PertV} for $m=1$.
We use the chain rule
  \begin{align*}
 \nab  \tilde \bv_h = \hat \nab (A_T \hat \bv_h) \circ \varphi_{\tilde T}^{-1} D\varphi_{\tilde T}^{-1},\quad
 \nab  \tilde{\breve{\bv}}_h = \hat \nab (\breve{A}_{\breve T} \hat{\breve{\bv}}_h) \circ \varphi_{\tilde T}^{-1} D\varphi_{\tilde T}^{-1},
\end{align*}
to conclude
\begin{align*}
|\tilde \bv_h- \tilde{\breve{\bv}}_h|_{H^1(\tilde T)}^2
& \approx h^d \int_{\hat T}\left | \hat \nab (A_T \hat \bv_h) (D\varphi_{\tilde T})^{-1} - 
 \hat \nab (\breve{A}_{\breve T} \hat{\breve{\bv}}_h) (D\varphi_{\tilde T})^{-1}\right|^2\dif \hat{\xb}\\
&\lesssim 
h^{d-2} \int_{\hat T}\left | \hat \nab (A_T \hat \bv_h) - \hat \nab (\breve{A}_{\breve T} \hat{\breve{\bv}}_h) \right|^2\dif \hat{\xb}.
 \end{align*}

We then use \eqref{eqn:ABounds} to obtain
 \begin{align*}
|\tilde \bv_h- \tilde{\breve{\bv}}_h|_{H^1(\tilde T)}^2
 &\lesssim h^{d-2}  \int_{\hat T}\left | \hat \nab (A_T \hat \bv_h-\breve{A}_{\breve T} \hat{\breve{\bv}}_h)\right|^2\dif \hat{\xb}\\
&\lesssim h^{d-2}  \int_{\hat T}\left | \hat \nab (A_T (\hat \bv_h-A^{-1}_T \breve{A}_{\breve T} \hat{\breve{\bv}}_h))\right|^2\dif \hat{\xb}\\
&\lesssim h^{d-2} \left(h^{4-2d}_T \int_{\hat T} \left| \hat \bv_h-A^{-1}_T \breve{A}_{\breve T} \hat{\breve{\bv}}_h\right|^2\dif \hat{\xb}
 +h^{2-2d}  \int_{\hat T} \left|\hat \nab( \hat \bv_h-A^{-1}_T \breve{A}_{\breve T} \hat{\breve{\bv}}_h)\right|^2\dif\hat{\xb}
 \right).
 \end{align*}
Recalling \eqref{eqn:hatvhInterp}, we apply the Bramble-Hilbert Lemma 
and \eqref{eqn:HONorm}:
 \begin{align*}
 |\tilde \bv_h- \tilde{\breve{\bv}}_h|_{H^1(\tilde T)}^2 \lesssim h^{-d} |A_T^{-1} \breve{A}_{\breve T} \hat{\breve{\bv}}_h|_{H^{k+1}(\hat T)}^2\lesssim h^2 \|\breve{\bv}_h\|_{H^1(\breve T)}^2.
 \end{align*}
\end{proof}

\begin{lem}\label{lem:qPert}
For every $q_h\in {Q}_h$, there exists
$\breve{q}_h\in \breve{Q}_h$ such that $q_h\circ \bPhi=\breve{q}_h\circ \breve \bPhi$.
Consequently,
\begin{align}\label{eqn:qStable}
\|q_h\|_{\Omega_h^{\calT}}\le C_1 \|\breve q_h\|_{\Omega_h^{\breve \calT}}.
\end{align}
\end{lem}
\begin{proof}
Given ${q}_h\in {Q}_h$, set $\breve{q}_h|_{\breve T} = {q}_h|_T\circ {\varphi}_T\circ \breve \varphi_{\breve T}^{-1}$ for each $\breve T\in \breve \calT_h$,
where $T = \bPhi_h \circ \breve\bPhi_h^{-1}(\breve T)$.
The result then follows upon noting that $\bPhi_h = \varphi_T\circ \varphi_{\tilde T}^{-1}$
and $\breve \bPhi_h = \breve \varphi_{\breve T} \circ \varphi_{\tilde T}^{-1}$.
The estimate \eqref{eqn:qStable} follows from a change of variables
and \eqref{PhiBounds}, \eqref{brevePhiBounds}.
\end{proof}

\begin{proof}[Proof of \cref{Th:fitted}]
Fix $q_h\in Q_h^0$,  let $\breve q_h\in \breve Q_h$ satisfy
$\breve q_h = q_h \circ \bPhi_h \circ \breve \bPhi_h^{-1}$ (cf.~\cref{lem:qPert}),
and set $\breve c = \frac1{|\Omega_h^{\breve \calT}|} \int_{\Omega_h^{\breve \calT}} \breve q_h\dif \xb$ so that $(\breve q_h-\breve c)\in \breve Q_h^0$.
A change of variables and the zero-mean property of $q_h$ show 
\[
\breve{c} = \frac{1}{|\Omega_h^{\breve \calT}|} \int_{\Omega_h^{\calT}} q_h (\det(D\breve \bPhi_h \circ \bPhi_h^{-1})\det(D\bPhi^{-1})-1)\dif \xb,
\]
and therefore by \eqref{eqn:detPhi} and \eqref{eqn:detBrevePhi},
we have
\begin{equation}
\label{cBound}
\|\breve c\|_{\Omega_h^{\breve \calT}}\le C_2 h \|q_h\|_{\Omega_h^{\calT}}.
\end{equation}

From \cref{lem:BreveStable}, there exists $\breve{\bv}_h\in \breve \bV_h^0$ satisfying
\begin{equation}\label{eqn:InfSupAgain}
\breve{\beta}
\|\breve q_h-\breve c\|_{\Omega_h^{\breve \calT}} \le   \frac{\int_{\Omega_h^{\breve \calT}} ({\rm div}\,\breve \bv_h) (\breve q_h-\breve c)\dif \xb}{\|\breve \bv_h\|_{H^1(\Omega_h^{\breve \calT})}}
= \frac{\int_{\Omega_h^{\breve \calT}} ({\rm div}\,\breve \bv_h) \breve q_h\dif \xb}{\|\breve \bv_h\|_{H^1(\Omega_h^{\breve \calT})}},
\end{equation}
where in the last step, we used that $\breve \bv_h$ vanishes on the boundary
of $\Omega_h^{\breve \calT}$, as well as the $\bH({\rm div})$-conformity
of $\breve\bV_h$ (cf.~\cref{lem:BreveStable}).

Let $\bv_h\in \bV_h$ satisfy \eqref{eqn:PertV}, 
and set $\tilde \bv_h = \bv_h\circ \bPhi_h$, $\tilde{\breve{\bv}}_h = \breve \bv_h\circ \breve \bPhi_h$,
and $\tilde q_h = q_h \circ \bPhi_h = \breve q_h\circ \breve \bPhi_h$. 
By a change of variables, \eqref{eqn:detPhi} and \eqref{eqn:detBrevePhi}, we have
\begin{align*}
&\int_{\Omega_h^{\calT}} ({\rm div}\,\bv_h) q_h\dif \xb - \int_{\Omega_h^{\breve \calT}} ({\rm div}\,\breve \bv_h) \breve q_h\dif \xb\\
&\qquad = \int_{\Omega_h^{\tilde \calT}} {\rm tr}(\nab \tilde \bv_h (D\bPhi_h)^{-1})  \tilde q_h
\det(D\bPhi_h)\dif \xb-\int_{\Omega_h^{\tilde \calT}} {\rm tr}(\nab \tilde {\breve \bv}_h(D\breve{\bPhi}_h)^{-1}) \tilde q_h \det(D\bPhi_h)\dif \xb\\
&\qquad \le \int_{\Omega_h^{\tilde \calT}} 
\left({\rm tr}(\nab \tilde \bv_h (D\bPhi_h)^{-1}) -{\rm tr}(\nab \tilde {\breve \bv}_h(D\breve{\bPhi}_h)^{-1})\right)\tilde q_h\dif \xb
+ h(\|\tilde \bv_h\|_{H^1_h(\Omega_h^{\tilde \calT})} +\|\tilde{\breve{\bv}}_h\|_{H^1_h(\Omega_h^{\tilde \calT})})\|\tilde q_h\|_{\Omega_h^{\tilde \calT}}\\
&\qquad \lesssim \int_{\Omega_h^{\tilde \calT}} \left({\rm div}\,\tilde \bv_h - {\rm div}\,\tilde{\breve{\bv}}_h\right) \tilde q_h\dif \xb
+ h(\|\tilde \bv_h\|_{H^1(\Omega_h^{\tilde \calT})} +\|\tilde{\breve{\bv}}_h\|_{H^1(\Omega_h^{\tilde \calT})})\|\tilde q_h\|_{\Omega_h^{\tilde \calT}}.
\end{align*}
Thus by \eqref{eqn:PertV} and \eqref{PhiBounds}, there exists $C_3>0$ such that
\begin{align}\label{diffABC}
\int_{\Omega_h^{\calT}} ({\rm div}\,\bv_h) q_h\dif \xb - \int_{\Omega_h^{\breve \calT}} ({\rm div}\,\breve \bv_h) \breve q_h \dif \xb
\le C_3 h\|\breve{\bv}_h\|_{H^1_h(\Omega_h^{\breve\calT})} \|q_h\|_{\Omega_h^{\calT}}.
\end{align}

Using the estimates \eqref{PhiBounds} and \eqref{brevePhiBounds}
we easily conclude $\|\breve \bv_h\|_{H^1(\Omega_h^{\breve \calT})}\approx \|\bv_h\|_{H^1(\Omega^{\calT}_h)}$.
Consequently, it follows from \eqref{eqn:InfSupAgain} and \eqref{diffABC}
that
\begin{align}\label{eqn:asdf}
\breve{\beta}
\|\breve q_h-\breve c\|_{\Omega_h^{\breve \calT}} 
&\le 
\frac{\int_{\Omega_h^{ \calT}} ({\rm div}\, \bv_h) q_h\dif \xb}{\|\breve \bv_h\|_{H^1(\Omega_h^{\breve \calT})}}
+C_3 h \|q_h\|_{\Omega_h^{\calT}}
\le C_4 \frac{\int_{\Omega_h^{ \calT}} ({\rm div}\, \bv_h) q_h\dif \xb}{\| \bv_h\|_{H^1(\Omega_h^{ \calT})}}
+C_3 h \|q_h\|_{\Omega_h^{\calT}}
\end{align}
for a constant $C_4>0$ independent of the mesh.

Finally, we conclude from \eqref{eqn:qStable}, \eqref{cBound}, and \eqref{eqn:asdf} that
\begin{align*}
\|q_h\|_{\Omega_h^{\calT}}
&\le C_1 \|\breve q_h\|_{\Omega_h^{\breve \calT}}\\
&\le C_1 (\|\breve q_h-\breve c\|_{\Omega_h^{\breve \calT}} + C_2 h\|q_h\|_{\Omega_h^{\calT}})\\
&\le C_1 C_4 \breve \beta^{-1} \frac{\int_{\Omega_h^{ \calT}} ({\rm div}\, \bv_h) q_h\dif \xb}{\| \bv_h\|_{H^1(\Omega_h^{ \calT})}}+ h(C_1 C_3 \breve \beta^{-1} +C_1C_2) \|q_h\|_{\Omega_h^{\calT}}.
\end{align*}
Thus, with $\beta = \breve \beta \left(1-h(C_1 C_3 \breve \beta^{-1} +C_1C_2)\right)C_1^{-1} C_4^{-1}$, the desired result holds for $h$ sufficiently small.
\end{proof}

\section{Proof of \texorpdfstring{\eqref{eqn:ABounds}}{}}\label{app-ABounds}
\begin{proof}
We prove the bounds for $A_T$, 
as the results for $\breve A_{\breve T}$ are identical.
We use the chain rule identity (cf.~\cite{Bernardi89})
\begin{align}\label{chainRule}
D^m (f_1\circ f_2) = \sum_{j=1}^m (D^j f_1)\circ f_2 \sum_{\alpha\in E(m,j)} c_\alpha \prod_{i=1}^m (D^i f_2)^{\alpha_i},
\end{align}
where $E(m,j) = \{\alpha\in \mathbb{N}_0^m:\ |\alpha| = j\text{ and } \sum_{j=1}^m j \alpha_j = m\}$,
and $c_\alpha\in \mathbb{R}$.

Note that $A_T^{-1} = {\rm adj}(D\varphi_T)$, the adjugate matrix
of $D\varphi_T$. Therefore, the entries of $A_T^{-1}$ are polynomials
of degree $(d-1)$ in $D\varphi_T$. We take $f_1 = x^{d-1}$ and $f_2$
to be the derivatives of $\varphi_T$ 
in \eqref{chainRule} to conclude
\begin{align*}
|A_T^{-1}|_{W^{m,\infty}(\hat T)}
&\lesssim \sum_{j=1}^{d-1} |\varphi_T|_{W^{1,\infty}(\hat T)}^{d-1-j} \sum_{\alpha\in E(m,j)} \prod_{i=1}^m |\varphi_T|_{W^{i+1,\infty}(\hat T)}^{\alpha_i}\\
&\lesssim \sum_{j=1}^{d-1} h^{d-1-j}  \sum_{\alpha\in E(m,j)} \prod_{i=1}^m h^{(i+1)\alpha_i}
\lesssim \sum_{j=1}^{d-1} h^{d-1-j}\cdot h^{m+j}\lesssim h^{d+m-1}.
\end{align*}

Next, a simple argument shows 
\begin{equation}
    \label{DdetvarphiT}
|\det(D\varphi_T)|_{W^{m,\infty}(\hat T)}\lesssim h^{d+m},
\end{equation}
and so
by \eqref{chainRule}
\begin{align*}
 \left|\frac1{\det(D\varphi_T)}\right|_{W^{m,\infty}(\hat T)}
&\lesssim \sum_{j=1}^m \|\det(D\varphi_T)\|_{L^\infty(\hat T)}^{-j-1} \sum_{\alpha\in E(m,j)} \prod_{i=1}^m |\det(D\varphi_T)|^{\alpha_i}_{W^{i,\infty}(\hat T)}\\
&\lesssim  \sum_{j=1}^m h^{-d(j+1)}  \sum_{\alpha\in E(m,j)} \prod_{i=1}^m h^{(d+i)\alpha_i}
\lesssim  \sum_{j=1}^m h^{-d(j+1)}\cdot h^{d j+m}\lesssim h^{m-d}.
\end{align*}
Finally, we apply the product rule to conclude
\begin{align*}
|A_T|_{W^{m,\infty}(\hat T)}
&\lesssim \sum_{\ell=0}^m  \left|\frac1{\det(D\varphi_T)}\right|_{W^{\ell,\infty}(\hat T)} |D\varphi_T|_{W^{m-\ell,\infty}(\hat T)}
\lesssim \sum_{\ell=0}^m  h^{\ell-d} \cdot h^{m-\ell+1}\lesssim h^{m-d+1}.
\end{align*}
\end{proof}

\section{Proof of \cref{lem:Interp}}
\begin{proof}
From \cite[Lemma 4.4]{DN24}, there exists
an interpolant $I^1_V:{\bm C}(\Omega_h^{\calT})\to \bV_h$ such that ($s=0,1$)
\begin{equation}\label{eqn:InterpEstimate}
 \|\bv -I^1_V\bv\|_{H^s(K)}\lesssim h^{k+1-s}\|\bv\|_{H^{k+1}(K)}\qquad \forall \bv\in \bH^{k+1}(K)\ \forall K\in \mathcal{T}^{\textup{Al}}_h. 
 \end{equation}

To construct $I_V \bv$, we map
$I^1_V \bv$ to the affine triangulation, enrich
this function with facet bubble functions,
and then map back to the computational triangulation.
To explain this procedure in further detail, 
we use the notation $\bv_h^1 = I^1_V \bv$ and set
$\tilde \bv_h^1 = \left((\det(D\Phi_h^{-1}))^{-1} D\Phi_h^{-1} \bv^1_h\right)\circ \bPhi_h$,
the Piola transform of $\bv_h^1$ with respect to $\bPhi_h^{-1}$.
We then see that $\tilde \bv^1_h\in {\bm H}({\rm div}, \Omega_h^{\tilde \calT})$ 
and is a piecewise polynomial
of degree $\le k$ with respect to $\Th$.
Likewise, we set $\tilde \bv = \left((\det(D\Phi_h^{-1}))^{-1} D\Phi_h^{-1} \bv\right)\circ \bPhi_h\in \bH({\rm div}, \Omega_h^{\tilde \calT})$.

For each $K\in \calT_h^{\textup{Al}}$ with ${\rm meas}_{d-1}(\p K\cap \p \Omega_h^{\calT})>0$,
let $\{F^i_K\}_{i=1}^{M_K}$ denote the set of facets of $K$ that
lie on the boundary $\p\Omega_h^{\calT}$ ($M_K\in \mathbb{N}$).
Set $\tilde K = \Phi_h^{-1}(K)\in \Th$, $\tilde F^i_K = \bPhi_h^{-1}(F^i_K)$, and let $\tilde b^K_i\in \mathbb{P}^d(\tilde K)$ denote 
the bubble function associated with $\tilde F^i_K$ that takes
the value $1$ at the barycenter of $\tilde F_K^i$ and vanishes
on $\p \tilde K\backslash \tilde F_K^i$. We also let $\tilde \bn_{F_i} = \tilde \bn_h|_{\tilde F_i}\in \bbR^d$, where $\tilde \bn_h$ denotes the outward unit normal
of $\p  \Omega_h^{\tilde \calT}$. 

We then set for each boundary element $K$,
\[
\tilde \bv_h|_{\tilde K} := \tilde \bv^1_h|_{\tilde K}+\sum_{i=1}^{M_K} \alpha_i \tilde b_i^K \tilde \bn_{F_i},\quad \text{where}\quad \alpha_i = \left(\int_{\tilde F_i} \tilde b_i^K\right)^{-1}\int_{\tilde F^i_K} (\tilde \bv -\tilde \bv^1_h)\cdot \tilde \bn_{F_i},
\]
and $\tilde \bv_h|_{\tilde K} = \tilde \bv^1_h|_{\tilde K}$
for $\tilde K\in\Th$
with ${\rm meas}_{d-1}(\p \tilde K\cap \p  \Omega_h^{\tilde \calT})=0$.
Finally, we set
\[
I_V\bv :=\left((\det(D\bPhi_h))^{-1} D\bPhi_h \tilde \bv_h\right)\circ \bPhi_h^{-1}.
\]

Because the facet bubbles vanish at interior degrees of freedom
and degrees of freedom that lie on $(d-2)$-dimensional simplices,
we easily conclude $I_V \bv\in \bV_h$. Moreover, by 
the normal-preserving properties of the Piola transform, we have
\begin{align*}
  \int_{F_K^i}        \! (I_V \bv)\cdot \bn_h\dif s
= \int_{\tilde F_K^i} \! \tilde \bv_h\cdot \tilde \bn_h\dif s
= \int_{\tilde F_K^i} \! \tilde \bv_h^1\cdot\tilde \bn_h\dif s + \alpha_i \int_{\tilde F_K^i}\! \tilde b_i^K\dif s
= \int_{\tilde F_i^K} \! \tilde \bv\cdot \tilde \bn_h\dif s
= \int_{F_K^i}        \! \bv\cdot \bn_h\dif s.
\end{align*}
Thus, 
\[
\int_{\p \Omega_h^{\calT}} (I_V \bv)\cdot \bn_h\dif s = \int_{\p \Omega_h^{\calT}} \bv\cdot \bn_h\dif s.
\]
Next, we note that
\begin{align*}
|\alpha_i| 
&\lesssim  |\tilde F_K^i|^{-\frac12}(h^{\frac12}|\tilde \bv - \tilde \bv^1_h|_{H^1(\tilde K)}+h^{-\frac12}\|\tilde \bv-\tilde \bv^1_h\|_{L^2(\tilde K)})\\
&\lesssim h^{-\frac{d-1}2}
(h^{\frac12}\| \bv -  \bv^1_h\|_{H^1( K)}+h^{-\frac12}\| \bv- \bv^1_h\|_{L^2(K)})
\lesssim h^{k+1-\frac{d}{2}}\|\bv\|_{H^{k+1}(K)}.
\end{align*}
Thus,
\begin{align*}
\|\bv - I_V\bv \|_{H^s(K)}
&\lesssim \|\tilde \bv - \tilde\bv_h \|_{H^s(\tilde K)}
\lesssim \| \bv - \bv^1_h \|_{H^s(K)}+h^{\frac{d}{2}-s}\sum_{i=1}^{M_K} |\alpha_i|\lesssim h^{k+1-s} \|\bv\|_{H^{k+1}(K)}.
\end{align*}

\end{proof}

\section{Proof of \cref{LAux}}

Recall that  $\tilde J_K=\det(D\restr{\bPhi_h}{\tilde K})$. Equation \eqref{eqn:detPhi}
implies that 
\begin{equation}\label{Jbound}
0<c\le \mathcal{E}^{\mathbb{P}}  \tilde J_K \le C \quad \text{on}~\omega(\tilde K),~~\forall\,\tilde K\in \Th,\ K = \bPhi_h(\tilde K),    
\end{equation}
with constants $c,\,C$  independent of $h$ and $\tilde K$.
Here $\omega(\tilde K)$ is a set of simplices touching $\tilde K$,
$\mathcal{E}^{\mathbb{P}}:\mathbb{P}^{m}\to \mathbb{P}^{m}$ is the canonical
global extension operator of polynomials to $\mathbb{R}^d$, and we recall $\varphi_{\tilde K}:\hat T\to \tilde K$ is an
affine mapping.
The  pointwise bounds in \eqref{Jbound} imply that for any $\tilde K\in \Th$ and $f\in\PP^{2k-1}(\tilde K)$ it holds
\begin{equation}\label{Extbound}
\begin{split}
\|\mathcal{E}^\mathcal{R}(\tilde J_K^{-1}  f) \|_{\omega(\tilde K)}
&\lesssim 
\|\tilde J_K^{-1} f \|_{\tilde K},\\
\|\nab\mathcal{E}^\mathcal{R}(\tilde J_K^{-1}  f) \|_{\omega(\tilde K)}
&\lesssim 
\|\nab(\tilde J_K^{-1} f) \|_{\tilde K}.    
\end{split}
\end{equation}
In the second inequality, we also used $\tilde J_K\in \PP^{d(k-1)}(\tilde K)$.

Estimate \eqref{aux702} follows from the bounds  \eqref{PhiBounds}--\eqref{eqn:detPhi}, the triangle and inverse inequalities, and \eqref{Extbound}:
\begin{equation*}
\begin{split}
     \|\nab \vh\|_{K_1}^2\lesssim
     \|\nab \vh\circ\bPhi_h\|_{\tilde K_1}^2&
     \lesssim \|\nab \mathcal{E}^\mathcal{R} \restr{\vh\circ\bPhi_h}{\tilde K_2}
     \|_{\tilde K_1}^2+
     \|\nab (\vh\circ\bPhi_h -\mathcal{E}^\mathcal{R} \restr{\vh\circ\bPhi_h}{\tilde K_2}) 
     \|_{\tilde K_1}^2  \\ 
     &\lesssim
     \|\nab \vh\circ\bPhi_h
     \|_{\tilde K_2}^2+
     \|\nabla(\vh -\mathcal{E}^\mathcal{R} (\vh\circ\bPhi_h|_{\tilde K_2}) \circ\bPhi_h^{-1})
     \|_{K_1}^2\\ 
&\lesssim
     \|\nab \vh
     \|_{K_2}^2+
     h^{-2}\|\jump{\vh}_\omega 
     \|_{\omega_F}^2.
\end{split}
\end{equation*}

Consider $\vb\in H^{k+1}(\Omega_h^\calT)$ and $F\in \Fhogp$.
Denote by $B(F)$ a minimal ball such that $\omega_F\subset B(F)$, so that $\text{diam}(B)\le Ch$ with some $C>0$ independent of $F$ and $h$, and thus
\[
\|\vb- \Pi_B\vb\|_{\omega_F}\le \|\vb- \Pi_B\vb\|_{B(F)}\lesssim h^{k+1}\|\bv\|_{H^{k+1}(B(F))},
\]
where $\Pi_B$ is the $L^2$-projection onto $\mathbb{P}^{k}(B(F))$.
We use this and \eqref{Extbound} to estimate for $\omega_F=K_1\cup K_2$: 
\begin{equation}\label{aux755}
\begin{split}
     \|\vb_p-\mathcal{E}^\mathcal{R}&( \restr{\vb_p\circ\bPhi_h}{\tilde K_2})\circ\bPhi_h^{-1}\|_{K_1}\le 
    \|\vb_p-\Pi_B\vb\|_{K_1} +\|\Pi_B\vb-\mathcal{E}^\mathcal{R} (\restr{\vb_p\circ\bPhi_h}{\tilde K_2})\circ\bPhi_h^{-1}\|_{K_1}\\
    &\lesssim 
    \|\vb_p-\Pi_B\vb\|_{K_1} +\|\Pi_B\vb\circ\bPhi_h-\mathcal{E}^\mathcal{R} (\restr{\vb_p\circ\bPhi_h}{\tilde K_2})\|_{\tilde K_1} \\
&\lesssim 
    \|\vb_p-\Pi_B\vb\|_{K_1} +\|\Pi_B\vb\circ\bPhi_h-\vb_p\circ\bPhi_h\|_{\tilde K_2} \lesssim 
    \|\vb_p-\Pi_B\vb\|_{\omega_F}\\ 
    &\lesssim 
    \|\vb_p-\bv\|_{\omega_F}+\|\bv-\Pi_B\vb\|_{\omega_F}\\ 
    & \lesssim \|\vb_p-\bv\|_{\omega_F}+ h^{k+1}\|\bv\|_{H^{k+1}(B(F))}.   
\end{split}
\end{equation}
For the  first norm on the right-hand side, we have
\begin{equation}
\label{aux770}
\begin{split}
\|\vb-\bv_p\|_{K_i}&\lesssim
\|\vb\circ\bPhi_h - \tilde J_{K_i} ^{-1} \Pi_{\tilde K_i}(\tilde J_{K_i} \vb\circ\bPhi_h)\|_{\tilde K_i}\\
&\lesssim
\|\tilde J_K  \vb\circ\bPhi_h - \Pi_{\tilde K_i}(\tilde J_{K_i}  \vb\circ\bPhi_h)\|_{\tilde K_i}\\
&\lesssim h^{k+1} |\tilde J_{K_i}  \vb\circ\bPhi_h|_{H^{k+1}(\tilde K_i)}\\
&\lesssim h^{k+1} \sum_{\ell=0}^{k+1}  |\tilde J_{K_i} |_{W^{\ell,\infty}(\tilde K_i)} |\vb\circ\bPhi_h|_{H^{k+1-\ell}(\tilde K_i)}\\
&\lesssim h^{k+1-d} \|\bv\|_{H^{k+1}(K_i)} \sum_{\ell=0}^{k+1} h^{-\ell} |\det(D\varphi_{K_i})|_{W^{\ell,\infty}(\hat  T)}\lesssim h^{k+1} \|\bv\|_{H^{k+1}(K_i)},
\end{split}
\end{equation}
where we used $\tilde J_{K_i} = c h^{-d} \det(D\varphi_{K_i})\circ \varphi_{\tilde K_i}^{-1}$ and \eqref{DdetvarphiT} in the last two inequalities.
Now \eqref{aux705} follows from substituting \eqref{aux770} in \eqref{aux755}, scaling by $h^{-2}$, summing up over all $F\in \Fhogp$ and using the finite overlap of $B(F)$.

The third estimate, equation \eqref{aux707}, follows from the straightforward application of the triangle inequality, stability of the polynomial extension and projection, and the interpolation property: Let $\eb=\vb-\pi_V\vb$
\[
\begin{split}
i_h(\eb,\eb)& = 
\sum_{F\in \Fhogp}h^{-2}\|\eb_p-\mathcal{E}^\mathcal{R}( \restr{\eb_p\circ\bPhi_h}{\tilde K_2})\circ\bPhi_h^{-1}\|_{\omega_F}^2
\lesssim
\sum_{K\in \Thogp}h^{-2}\|\eb_p\|_{K}^2\\
&=\sum_{K\in \Thogp}h^{-2}\|\tilde J_K^{-1}\Pi_{\tilde K}(\tilde J_K\eb\circ\bPhi_h)\circ\bPhi_h^{-1}\|_{K}^2
\lesssim \sum_{K\in \Thogp}h^{-2}\|\eb\|_{K}^2.
\end{split}
\]
To estimate the right-hand side, we use the definition of  $\pi_V\vb$ from \cref{lem:Fortin}, i.e. $\pi_V\vb=I_V\vb+\bw_h$ and split 
\[
\sum_{K\in \Thogp}h^{-2}\|\eb\|_{K}^2\lesssim \sum_{K\in \Thogp}h^{-2}\|\vb-I_V\vb\|_{K}^2+ \sum_{K\in \Thogp}h^{-2}\|\bw_h\|_{K}^2.
\]
The estimate 
$\|\vb-I_V\vb\|_{\OT}^2\lesssim h^{2k+2} \|\vb\|^2_{H^{k+1}(\Omega_h^\calT)}$
comes from \cref{lem:Interp}. For $\bw_h$, we use $\bw_h=0$ on $\partial\OT$ and the Poincaré inequality for the $O(h)$-narrow band of simplices from $\Thogp$:
\[
\sum_{K\in \Thogp}h^{-2}\|\bw_h\|_{K}^2\lesssim
\|\nabla\bw_h\|_{\OT}^2 \lesssim \frac1{\beta}\tnrma{\bv-I_V \bv}^2\lesssim
h^{2k} \|\vb\|^2_{H^{k+1}(\Omega_h^\calT)}.
\]

\printbibliography

\end{document}